\documentclass{article}

\usepackage[backend=bibtex,style=numeric,sorting=none,maxcitenames=2]{biblatex}
\usepackage{amsmath}
\usepackage{amssymb}
\usepackage{amsthm}
\usepackage{xfrac}
\usepackage{mathtools}
\usepackage{graphicx}
\usepackage{subcaption}
\usepackage{comment}
\usepackage{xcolor}
\usepackage{float}
\usepackage{hyperref}

\theoremstyle{plain}
\newtheorem{theorem}{Theorem}
\newtheorem{lemma}[theorem]{Lemma}
\newtheorem{corollary}{Corollary}[theorem]
\newtheorem{proposition}[theorem]{Proposition}
\theoremstyle{remark}
\newtheorem{remark}[theorem]{Remark}

\title{Exact barycentric rational interpolation of analytic functions with an application to the Bessel functions}
\author{Brett R. Green}
\begin{document}

	\maketitle
	
	\begin{abstract}
		We derive a method of analytic function expansion using barycentric rational interpolation, in which an arbitrary analytic weight function defines the interpolation and nodes. Truncating these expansions yields successive approximations, and the derivation of the expansions gives the approximation error as a byproduct. We demonstrate the method with cosine and the Bessel functions of the first kind for several choices of weight functions, and we examine the accuracy and convergence of the approximate interpolations as a function of the weight and the number of nodes used in the interpolation.
	\end{abstract}
	
	\section{Introduction}
	
	The most common function expansions for analytic functions, such as Taylor series and Pad\'{e} approximants, show that an analytic function can be approximated arbitrarily well with more and more information about its behavior at a single point.
	With complete information at that point, in the form of infinitely many terms, the expansions become exact.
	Conversely, it is also possible to produce approximations with limited information at each of many points, and to make these approximations exact by taking the limit of infinitely many points.
	Weierstrauss products expand functions as products of their zeros, providing one example of such a "nonlocal" function expansion.
	
	Another candidate for such an expansion and approximation comes from barycentric rational interpolation. An interpolation of this type is based on a set of $m$ nodes $z_k$ at which data is provided and generally has the form
	\begin{equation} \label{eq:barycentric-rational-interpolation-general-form}
		F(z) = \frac{ \sum_{k=1}^m r_k(z)p_k(z) }{ \sum_{k=1}^m r_k(z) } \, ,
	\end{equation}
	where each $r_k(z)$ is a rational function with a pole at $z_k$ and each $p_k(z)$ is a polynomial.\cite{floater_barycentric_2007,nakatsukasa_aaa_2018,berrut_recent_2005,berrut_recent_2014,mitchell_two_2025} $r_k(z)$ act as weighting functions which diverge at $z_k$ and thus anchor $f(z)$ to the value $f(z_k)=p_k(z_k)$ at $z_k$, hence the name "barycentric."
	These interpolations have proven advantageous in terms of stability and convergence with increasing density of nodes, but they have typically been explored in terms of interpolating arbitrary data rather than as a method of exact expansion.
	A method due to Berrut based on the theory of cardinal functions provides an exact barycentric expansion for interpolation in the special case of equally spaced nodes.\cite{whittaker_xviiifunctions_1915,berrut_barycentric_1989}
	The aim of this paper is to demonstrate a method of exact expansion using barycentric rational interpolation with more general node placement and weighting.
	When the resulting expansions are truncated to finitely many nodes, the same calculation used to derive the expansion also provides bounds on the error of the finite approximations.
	
	The core of the idea is to take the analytic function to be interpolated, $F(z)$, choose an analytic weight function, $G(z)$, and write
	\begin{equation} \label{eq:intro-F-G-split}
		F(z) = \frac{\sfrac{F(z)}{G(z)}}{\sfrac{1}{G(z)}} \, .
	\end{equation}
	The zeros of $G(z)$ become the nodes for this interpolation and we control their nature and placement through the choice of $G(z)$. We then render this expression in barycentric rational form by using residue theory to expand the meromorphic functions $f^{(num)}(z) = \sfrac{F(z)}{G(z)}$ and $f^{(den)}(z) = \sfrac{1}{G(z)}$ in the form
	\begin{equation} \label{eq:Mittag-Leffler-expansion-intro}
		f(z) = \sum_{k=1}^\infty \sum_{n=1}^{o_k} \frac{a_{-n}^{(k)}}{\left(z-z_k\right)^n} + g(z) \, ,
	\end{equation}
	where $g(z)$ is an analytic function and $o_k$ is the order of $f(z)$'s pole at $z_k$. This idea of separating a meromorphic function into a principal part and an analytic part is familiar from Mittag-Leffler's theorem.\cite{marshall_complex_2019}
	The result
	\begin{equation} \label{eq:gen-bary-rat-interp-prototype}
		F(z)
		= \frac{\sfrac{F(z)}{G(z)}}{\sfrac{1}{G(z)}}
		= \frac{\sum_{k=1}^\infty \sum_{n=1}^{o_k^{(num)}} \frac{a_{-n}^{(k,num)}}{\left(z-z_k\right)^n} + g^{(num)}(z)}
		{\sum_{k=1}^\infty \sum_{n=1}^{o_k^{(den)}} \frac{a_{-n}^{(\alpha,den)}}{\left(z-z_k\right)^n} + g^{(den)}(z)}
	\end{equation}
	fits the barycentric rational form of Eq.~(\ref{eq:barycentric-rational-interpolation-general-form}) except for the presence of $g^{(num)}(z)$, $g^{(den)}(z)$, which we will call the \emph{analytic remainders}.
	The Laurent coefficients $a_{-n}^{(k,num)}$, $a_{-n}^{(k,den)}$ can be found with standard methods and the analytic remainders can be calculated as power series via contour integration.
	Calculating the analytic remainders is the main challenge in applying this method, but asymptotic methods and straightforward bounds can bring it within reach.
	Additionally, the contour integrals provide bounds on the error of approximate interpolations in which only finitely many nodes are used.
	
	The paper is outlined as follows.
	We show how to construct the desired expansions and highlight a few of their properties in Sec.~\ref{sec:barycentric-construction}.
	As a simple demonstration, we apply this method to cosine in Sec.~\ref{sec:cos-barycentric}.
	We then show how to apply the method to the more interesting case of the Bessel functions of the first kind, $J_q(z)$, to yield closed-form barycentric rational interpolation formulae in Sec.~\ref{sec:bessel-barycentric}.
	
	\section{Constructing barycentric rational expansions} \label{sec:barycentric-construction}
	
	We will take the constcuction of these barycentric expansions in two steps. In the first step, we establish the method for separating meromorphic functions into a sum over principal parts and an analytic remainder in the form of Eq.~(\ref{eq:Mittag-Leffler-expansion-intro}). In the second step, we apply this result to produce the barycentric expansion as given Eq.~(\ref{eq:gen-bary-rat-interp-prototype}) and examine its properties.
	
	\subsection{Expanding meromorphic functions as a sum over principal parts and an analytic remainder} \label{sec:general-analytic-remainder-expression}
	
	In the following lemma, the region $\Omega_m$ should be thought of as the region over which the nodes are distributed. To get the exact barycentric expansion, we will take the limit of a sequence of $\Omega_m$ with each $\Omega_m \subset \Omega_{m+1}$ and $\lim_{m \to \infty} \Omega_m = \mathbb{C}$.
	
	\begin{lemma} \label{lemma:expansion-lemma}
		Let $f(z)$ be a meromorphic function defined on $\mathbb{C}\backslash\{z_k\}$ with discrete poles $z_k$, and let $\Omega_m$ be a simply connected set whose boundary $\partial\Omega_m$ does not pass through any poles of $f(z)$. Then, provided $z\in\Omega_m$ is not a pole of $f(z)$,
		\begin{equation} \label{eq:f-with-g}
			f(z) = \sum_{z_k \in \Omega_m} \sum_{n=1}^{o_k} \frac{a_{-n}^{(k)}}{(z-z_k)^n} + g_m(z) \, .
		\end{equation}
		where $o_k$ is the degree of the pole at $z_k$ and
		\begin{equation} \label{eq:g-integral}
			g_m(z) = \frac{1}{2 \pi i} \oint_{\partial\Omega_m} dw \frac{f(w)}{w-z} \, .
		\end{equation}
		Furthermore, if $|z| < \max_{w\in\partial\Omega} |w|$, then $g(z)$ admits a power series expansion with coefficients which are independent of $z$,
		\begin{equation} \label{eq:g-integral-power-series}
			g_m(z) = \sum_{j=0}^\infty z^j \left( \frac{1}{2 \pi i} \oint_{\partial\Omega_m} dw \frac{f(w)}{w^{j+1}} \right) \, .
		\end{equation}
	\end{lemma}
	
	\begin{proof}
		The poles of $\frac{f(w)}{w-z}$ as a function of $w$ are $w=z$ and the poles of $f(w)$ itself, i.e., $w=z_k$. Since $z$ is not a pole of $f(w)$,
		\begin{equation}
			Res\left[\frac{f(w)}{z-w} , z\right] = -f(z) \, .
		\end{equation}
		For the poles at each $z_k$, use the limit formula for residues with the Laurent expansion of $f(w)$ about $z_k$,
		\begin{align}
			Res\left[\frac{f(w)}{z-w} , z_k\right] & = \sum_{n=1}^{o_k} Res\left[\frac{1}{z-w}\frac{a_{-n}^{(k)}}{(w-z_k)^n} , z_k\right] \\
			& = \sum_{n=1}^{o_k} \frac{a_{-n}^{(k)}}{(n-1)!} \lim_{w \to z_k} \left(\frac{d^{n-1}}{dw^{n-1}}\left( (w-z_k)^n \frac{1}{z-w} \frac{1}{(w-z_k)^n} \right)\right) \nonumber \\
			& = \sum_{n=1}^{o_k} \frac{a_{-n}^{(k)}}{(n-1)!} \lim_{w \to z_k} \left( (n-1)! \frac{1}{(z-w)^n} \right) = \sum_{n=1}^{o_k} \frac{a_{-n}^{(k)}}{(z-z_k)^n} \, , \nonumber
		\end{align}
		where $a_{-n}^{(k)}$ is the $-n$th term in the principal part of the Laurent expansion of $f(z)$ about its pole at $z_k$. Hence,
		\begin{equation} \label{eq:principal-part-integral}
			\frac{1}{2 \pi i} \oint_{\partial\Omega_m} dw \frac{f(w)}{z-w} = - f(z) + \sum_{z_k \in \Omega_m} \sum_{n=1}^{o_k} \frac{a_{-n}^{(k)}}{(z-z_k)^n} \, ,
		\end{equation}
		which can be rearranged to yield Eq.~(\ref{eq:f-with-g}).
		
		If $|z| < \max_{w\in\partial\Omega_m} |w|$, we can separate out dependence on $z$ by expanding in powers of z,
		\begin{equation}
			\oint_{\partial\Omega_m} dw \frac{f(w)}{w-z} = \oint_{\partial\Omega_m} dw \frac{1}{w} \frac{f(w)}{1-z/w}  = \sum_{j=0}^\infty z^j \oint_{\partial\Omega_m} dw \frac{f(w)}{w^{j+1}} \, .
		\end{equation}
		We are able to exchange the order of the sum and integral because $f(w)$ has a maximum on $\partial\Omega_m$ and the geometric series converges uniformly on $\partial\Omega_m$. Taking this expansion in Eq.~(\ref{eq:g-integral}) yields Eq.~(\ref{eq:g-integral-power-series}).
	\end{proof}
	
	All the preceding formulae are exact and give us a way to express $f(z)$ in terms of a sum over its Laurent series principal parts at each pole and an analytic remainder, both of which depend on the size of the integration region $\Omega_m$. Enlarging the region to encapsulate more poles adds new terms to the principal part sum in a way that is easy to evaluate. On the other hand, it is impractical, if not impossible, to evaluate the analytic remainder integrals for arbitrary $\partial\Omega_m$ - in fact, in our applications, these integrals can only be evaluated in closed form in the limit $\Omega_m\to\mathbb{C}$.
	
	The solution is to make the following observation.
	\begin{corollary} \label{cor:full-expansion-f}
		We can define the full expansion of a function $f(z)$ as the case $m\to\infty, \Omega_m\to\mathbb{C}$ of Lemma~\ref{lemma:expansion-lemma},
		\begin{equation} \label{eq:f-full-expansion}
			f(z) = \sum_{k=1}^\infty \sum_{n=1}^{o_k} \frac{a_{-n}^{(k)}}{(z-z_k)^n} + g(z) \, ,
		\end{equation}
		where $g(z)$ is the full analytic remainder defined by
		\begin{equation} \label{eq:full-analyt-rem-limit}
			g(z) = \lim_{m\to\infty} g_m(z) = \sum_{j=0}^\infty z^j \left( \frac{1}{2 \pi i} \lim_{m\to\infty} \oint_{\partial\Omega_m} dw \frac{f(w)}{w^{j+1}} \right)
		\end{equation}
		provided the limits of these integrals converge and produce a convergent series as $m\to\infty$ for all $j$.
	\end{corollary}
	
	\begin{proof}
		The interchanging of limits follows from dominated convergence. This approach is valid for all $z$ because $\lim_{m\to\infty} \Omega_m = \mathbb{C}$ and so we can always choose the first region in the sequence of $\Omega_m$ to be large enough to contain $z$. Likewise, to perform the series expansion, we can also choose the first region in the sequence to be large enough to satisfy $|z| < \max_{w\in\partial\Omega_m} |w|$.
	\end{proof}
	
	We can approximate the full expansion by including finitely many nodes but using the full analytic remainder. We will label such an approximation as
	\begin{equation} \label{eq:approximate-expansion-m}
		f_m(z) = \sum_{z_k\in\Omega_m} \sum_{n=1}^{o_k} \frac{a_{-n}^{(k)}}{(z-z_k)^n} + g(z) \, .
	\end{equation}
	While the error of these approximate expansions comes from missing terms in the sum over Laurent series principal parts, the error may also be calculated in terms of the exact and approximate analytic remainders.
	
	\begin{proposition} \label{prop:f-error-two-expressions}
		The error of the expansion in Eq.~(\ref{eq:approximate-expansion-m}),
		\begin{equation} \label{eq:f-error}
			e_m(z) = f_m(z) - f(z) \, ,
		\end{equation}
		can be written as either
		\begin{equation}
			e_m(z) = \sum_{z_k \notin \Omega_m} \sum_{n=1}^{o_k} \frac{a_{-n}^{(k)}}{\left(z-z_k\right)^n}
			\quad \text{or} \quad
			e_m(z) = g(z) - g_m(z) \, .
		\end{equation}
	\end{proposition}
	
	\begin{proof}
		Subtracting Eq.~(\ref{eq:f-full-expansion}) from Eq.~(\ref{eq:approximate-expansion-m}),
		\begin{align} \label{eq:f-error-Laurent-sum}
			e_m(z)
			& = \left(\sum_{z_k\in\Omega_m} \sum_{n=1}^{o_k} \frac{a_{-n}^{(k)}}{(z-z_k)^n} + g(z)\right) - \left(\sum_{k=1}^\infty \sum_{n=1}^{o_k} \frac{a_{-n}^{(k)}}{(z-z_k)^n} + g(z)\right) \nonumber \\
			& = \sum_{z_k\notin\Omega_m} \sum_{n=1}^{o_k} \frac{a_{-n}^{(k)}}{(z-z_k)^n} \, .
		\end{align}
		On the other hand, Eq.~(\ref{eq:g-integral-power-series}) is also exact for each $\Omega_m$. Subtracting this case from Eq.~(\ref{eq:approximate-expansion-m}),
		\begin{align}
			e_m(z)
			& = \left(\sum_{z_k\in\Omega_m} \sum_{n=1}^{o_k} \frac{a_{-n}^{(k)}}{(z-z_k)^n} + g(z)\right) - \left(\sum_{z_k\in\Omega_m} \sum_{n=1}^{o_k} \frac{a_{-n}^{(k)}}{(z-z_k)^n} + g_m(z)\right) \nonumber \\
			& = g(z) - g_m(z)
		\end{align}
	\end{proof}
	Using the approximate analytic remainder $g_m(z)$ will prove to be a much simpler way to analyze error than attempting to evaluate partial sums.
	Since we need to calculate $g_m(z)$ anyway in order to find the exact analytic remainder through the limit $g(z) = \lim_{m\to\infty} g_m(z)$, this expression for the expansion error also takes no extra effort to obtain.
	With this, we have everything necessary to construct barycentric rational expansions.
	
	\subsection{Exact and approximate barycentric rational expansions in practice}
	
	Recalling Eq.~(\ref{eq:gen-bary-rat-interp-prototype}), we apply Lemma~\ref{lemma:expansion-lemma} to expand $f^{(num)}(z) = \sfrac{F(z)}{G(z)}$ and $f^{(den)}(z) = \sfrac{1}{G(z)}$ and from these obtain the exact expansion for $F(z)$, and we use definition Eq.~(\ref{eq:approximate-expansion-m}) in the same way to produce the corresponding approximation $F_m(z)$,
	\begin{align}
		& F(z)
		= \frac{\sfrac{F(z)}{G(z)}}{\sfrac{1}{G(z)}}
		= \frac{f^{(num)}(z)}{f^{(den)}(z)}
		= \frac{\sum_{k=1}^\infty \sum_{n=1}^{o_k^{(num)}} \frac{a_{-n}^{(k,num)}}{\left(z-z_k\right)^n} + g^{(num)}(z)}
		{\sum_{k=1}^\infty \sum_{n=1}^{o_k^{(den)}} \frac{a_{-n}^{(\alpha,den)}}{\left(z-z_k\right)^n} + g^{(den)}(z)} \label{eq:gen-bary-rat-interp-exact} \\
		& \approx F_m(z) = \frac{f^{(num)}_m(z)}{f^{(den)}_m(z)}
		= \frac{\sum_{z_k\in\Omega_m} \sum_{n=1}^{o_k^{(num)}} \frac{a_{-n}^{(k,num)}}{\left(z-z_k\right)^n} + g^{(num)}(z)}
		{\sum_{z_k\in\Omega_m} \sum_{n=1}^{o_k^{(den)}} \frac{a_{-n}^{(\alpha,den)}}{\left(z-z_k\right)^n} + g^{(den)}(z)} \, . \label{eq:gen-bary-rat-interp-approx}
	\end{align}
	The only difference between Eq.~(\ref{eq:gen-bary-rat-interp-exact}) and Eq.~(\ref{eq:gen-bary-rat-interp-approx}) is in the sum over Laurent series principal parts. In the exact expression, all poles $z_k$ (now serving as the nodes of the barycentric rational interpolation) are included in the sum over Laurent series principal parts, while in the approximate expression, only those in $\Omega_m$ are.
	
	The approximate barycentric interpolations have the nice property of fixing not only the value of $F(z)$ at each node but also a certain number of its derivatives. In particular, recall that the nodes $z_k$ are the zeros of the weight function $G(z)$. The order of $G(z)$'s zero at $z_k$ dictates how many derivatives of $F_m(z)$ are guaranteed to be correct at $z_k$.
	
	\begin{proposition} \label{prop:anchor-point-fixed-derivative}
		A barycentric rational interpolation of the form in Eq.~(\ref{eq:gen-bary-rat-interp-approx}) will correctly give the first $o_k^{(den)}-1$ derivatives at each node $z_k$.
	\end{proposition}
	
	\begin{proof}
		Label the nodes $z_1, z_2, \dots z_m$ and expand the approximation of Eq.~(\ref{eq:gen-bary-rat-interp-approx}) about the node $z_1$ without loss of generality for small $\delta$.
		\begin{equation}\begin{aligned}
				F_m (z_1 + \delta) & = \frac{\sum_{n=1}^{o_1^{(num)}} \frac{a_{-n}^{(1,num)}}{\delta^n} + \sum_{k=2}^m \sum_{n=1}^{o_k^{(num)}} \frac{a_{-n}^{(k,num)}}{\left(z_1-z_k+\delta\right)^n} + g^{(num)}(z)}
				{\sum_{n=1}^{o_1^{(den)}} \frac{a_{-n}^{(1,den)}}{\delta^n} + \sum_{k=2}^m \sum_{n=1}^{o_k^{(den)}} \frac{a_{-n}^{(k,den)}}{\left(z_1-z_k+\delta\right)^n} + g^{(den)}(z)} \\
				& = \frac{\sum_{n=1}^{o_1^{(num)}} a_{-n}^{(1,num)} \delta^{o_1^{(den)}-n} + O\left(\delta^{o_1^{(den)}}\right)}
				{\sum_{n=1}^{o_1^{(den)}} a_{-n}^{(1,den)} \delta^{o_1^{(den)}-n} + O\left(\delta^{o_1^{(den)}}\right)} \\
				& = F_m(z_1) + \frac{dF_m}{dz}(z_1) \delta + \frac{1}{2} \frac{d^2 F_m}{dz^2}(z_1) \delta^2 + \dots
		\end{aligned}\end{equation}
		If we match powers of $\delta$ between the last two equations, we will find that the Laurent coefficients at $z_1$, that is, $a_{-n}^{(1,num)}$ and $a_{-n}^{(1,den)}$, are the only coefficients that appear at orders of $\delta$ lower than $o_1^{(den)}$. In other words, the first $o_1^{(den)}-1$ terms in the power series, which give the first $o_1^{(den)}-1$ derivatives of $F_m(z)$ at $z_1$, depend only on the inclusion of the node $z_1$; they do not depend on the inclusion of any other nodes.
		
		In the $m\to\infty$ case, the expansion includes all nodes and becomes exact, in which case its derivatives at $z_1$ must be correct. However, we just showed that these derivatives, and in particular their correctness, cannot depend on the inclusion of any other nodes. It follows that the first $o_1^{(den)}-1$ derivatives must always be correct whenever $z_1$ is included as an node producing a pole of order $o_1^{(den)}$.
	\end{proof}
	
	Prop.~\ref{prop:anchor-point-fixed-derivative} gives an intuitive explanation for how the choice of a weight function $G(z)$ can affect the accuracy of an approximate interpolation by fixing higher-order derivatives.
	The choice of weight function and the number of nodes included are the two determining factors in the accuracy of an approximate interpolation, and we will now turn to studying these more quantitatively.
	The result allows us to understand convergence easily in terms of the analytic remainder and can also be used to formalize the intuition about higher-order weight functions producing more accurate interpolations.
	
	\begin{proposition} \label{prop:F-error-convergence-rate}
		The rate of convergence for an approximate interpolation, defined through the error
		\begin{equation}
			E_m(z) = F_m(z) - F(z) \, ,
		\end{equation}
		is given by the rate of convergence for either the numerator analytic remainder or the denominator analytic remainder, whichever is slower. That is, considering proportinality with respect to $m$, either
		\begin{equation}
			E_m(z) \propto g^{(num)}(z) - g_m^{(num)}(z) \quad \text{or} \quad E_m(z) \propto g^{(den)}(z) - g_m^{(den)}(z) \, .
		\end{equation}
	\end{proposition}
	
	\begin{proof}
		Using Eqs.~(\ref{eq:gen-bary-rat-interp-exact}) and (\ref{eq:gen-bary-rat-interp-approx}) to write $E_m(z)$ in terms of the error in the expansion of its numerator and denominator from Eq.~(\ref{eq:f-error}), and keeping only to first order in $e_m^{(num)}(s)$ and $e_m^{(den)}(s)$, we have
		\begin{equation}\begin{aligned} \label{eq:F-error-with-e}
				E_m(z) & = F_m(z) - F(z)
				=  \frac{f_m^{(num)}(z)}{f_m^{(den)}(z)} - \frac{f^{(num)}(z)}{f^{(den)}(z)} \\
				& = \frac{f^{(num)}(z) + e_m^{(num)}(z)}{f^{(den)}(z) + e_m^{(den)}(z)} - \frac{f^{(num)}(z)}{f^{(den)}(z)} \\
				& \approx \frac{1}{f^{(den)}(z)}\left(e_m^{(num)}(z) - \frac{f^{(num)}(z)}{f^{(den)}(z)} e_m^{(den)}(z)\right) \\
				& = G(z)\left(e_m^{(num)}(z) - F(z)e_m^{(den)}(z)\right) \, .
		\end{aligned}\end{equation}
		The factor of $G(z)$ goes to $0$ at each node $z_k$, representing the property that the interpolation passes through each node exactly. Putting this aside, the interpolation error $E_m(z)$ only depends on $m$ through the numerator and denominator expansion errors, given through Prop.~\ref{prop:f-error-two-expressions} by
		\begin{equation}
			e_m^{(num)}(s) = g^{(num)}(z) - g_m^{(num)}(z) \, , \quad e_m^{(den)}(s) = g^{(den)}(z) - g_m^{(den)}(z) \, ,
		\end{equation}
		Hence, whichever analytic remainder integral converges the most slowly dictates the overall rate of convergence for $F_m(z)$ as $m$ increases and more nodes are added to the interpolation.
	\end{proof}
	
	\begin{remark} \label{remark:G-power-error-suppression}
		The factor of $G(z)$ in Eq.~(\ref{eq:F-error-with-e}) also explains quantitatively why a higher-order weight function which fixes more derivatives tends to produce a more accurate interpolation.
		The natural way to increase the order of the weight function's zeros at each node is to exponentiate it.
		
		Consider therefore the two interpolations
		\begin{equation} \label{eq:F-p-example-for-error-suppression}
			F^1(z) = \frac{\sfrac{F(z)}{G(z)}}{\sfrac{1}{G(z)}} \, ,
			\quad\quad
			F^p(z) = \frac{\sfrac{F(z)}{\left(G(z)\right)^p}}{\sfrac{1}{\left(G(z)\right)^p}} \, ,
		\end{equation}
		for integer $p>1$.
		Within some region $R$, we can assume without loss of generality that $\left|G(z)\right| \le 1$ because Eq.~(\ref{eq:F-p-example-for-error-suppression}) does not change when we replace ${G(z) \to \sfrac{G(z)}{\max_{z \in R} G(z)}}$. But if $\left|G(z)\right| \le 1$ in $R$, then ${\left|\left(G(z)\right)^p\right| \le |G(z)}|$ in $R$.
		
		Therefore, raising $G(z)$ to a higher power shrinks it as a prefactor in Eq.~(\ref{eq:F-error-with-e}) for the error $E_m(z)$.
		Moreover, because $R$ was arbitrary, this principle holds true for all $z\in\mathbb{C}$.
		While exponentiating the weight function usually reduces error in this way, it is not guaranteed to do so because changing the weight function also changes $e_m^{(num)}(s)$ and $e_m^{(den)}(s)$, and these may increase.
	\end{remark}
	
	From a practical perspective, there is a tradeoff between accuracy and simplicity in choosing a power on the weight function. While higher powers tend to produce more accurate interpolations, the interpolation formulae are more complex because higher-order poles yield Laurent series principal parts with more terms.
	
	Now it is time to put this method into practice. The key to applying it is in choosing use easy-to-integrate contours. In this paper, we will use square contours with sides along $x=\pm R_m$ or $y=\pm R_m$ and we will label these sides as $\Gamma_m^{(S)}$ where $S$ is $R$ for right, $T$ for top, $L$ for left, or $B$ for bottom. Separating the contour into four integrals along each side, we define
	\begin{equation} \label{eq:I-integral-def}
		I_{(S,m)}^{(j)} = \int_{\Gamma_m^{(S)}} dw \frac{f(w)}{w^{j+1}}
	\end{equation}
	and write
	\begin{equation} \label{eq:contour-integral-in-sides}
		\oint_{\partial\Omega_m} dw \frac{f(w)}{w^{j+1}}
		= I_{(R,m)}^{(j)} + I_{(T,m)}^{(j)} + I_{(L,m)}^{(j)} + I_{(B,m)}^{(j)} \, .
	\end{equation}
	We will make use of the symmetry of this contour as well as the symmetry of the interpolated function using the following argument.
	
	\begin{proposition} \label{prop:analytic-remainder-parity-argument}
		Take $\partial\Omega_m$ to be a square with sides at $x,y = \pm R_m$ and suppose that $f(w)$ has (anti)symmetry defined by $f(-w) = (-1)^p f(w)$, i.e., $p=0$ if $f$ is even and $p=1$ if $f$ is odd.
		Then the entire contour can be calculated from only the right side and top side integrals.
		Furthermore, if $j$ and $p$ have opposite parity, then opposite sides of the contour cancel and the entire contour integral $\oint_{\partial\Omega_m} dw \frac{f(w)}{w^{j+1}}$ vanishes.
	\end{proposition}
	
	\begin{proof}
		Changing variables $y\to-y$ and using the parity property,
		\begin{equation}\begin{aligned} \label{eq:parity-argument-RL}
				I_{(L,m)}^{(j)}
				& = \int_{R_m}^{-R_m} idy \frac{f(-R_m+iy)}{(-R_m+iy)^{j+1}}
				= \int_{-R_m}^{R_m} (-idy) \frac{f(-(R_m+iy))}{(-(R_m+iy))^{j+1}} \\
				& = (-1)^{p+j} \int_{\Gamma_m^{(R)}} dw \frac{f(w)}{w^{j+1}}
				= (-1)^{p+j} I_{(R,m)}^{(j)} \, .
		\end{aligned}\end{equation}
		Hence, if $p$ and $j$ have opposite parity, the two parts of the contour integral cancel.
		A similar calculation for the top and bottom sides of the contour yields the same result.
	\end{proof}
	
	This argument can readily be generalized to integrate $\frac{f(w)}{w^{j+1}}$ over any pair of symmetric contours $\gamma_+$, $\gamma_-$ parametrized as $\gamma_-(s) = -\gamma_+(1-s)$, but for simplicity we only focus on the necessary case here.
	e
	\section{Barycentric rational expansions for cosine} \label{sec:cos-barycentric}
	
	Expansions for trigonometric functions in the form of Eq.~(\ref{eq:Mittag-Leffler-expansion-intro}) and Sec.~\ref{sec:general-analytic-remainder-expression} are well-known and often used to demonstrate Mittag-Leffler's theorem, though not in the context of barycentric interpolation.\cite{arfken_mathematical_2013,marshall_complex_2019}
	These cases can be obtained quickly using Liouville's theorem, but for the purposes of demonstration, we derive them here using the more general method presented in Sec.~\ref{sec:barycentric-construction}.
	
	\subsection{Expansion weighted by $\sin^2 z$} \label{sec:first-cos-expansion}
	
	We take as an example
	\begin{equation}
		F(z) = \cos(z) \, , \quad G(z) = \left(\frac{d}{dz}\cos(z)\right)^2 = \sin^2(z)
	\end{equation}
	in Eq.~(\ref{eq:intro-F-G-split}) so that
	\begin{equation} \label{eq:cosine-barycentric-idea}
		\cos(z)
		= \frac{\sfrac{F(z)}{G(z)}}{\sfrac{1}{G(z)}}
		= \frac{\sfrac{\cos(z)}{\left(\frac{d}{dz}\cos(z)\right)^2}}{\sfrac{1}{\left(\frac{d}{dz}\cos(z)\right)^2}}
		= \frac{\sfrac{\cos(z)}{\sin^2(z)}}{\sfrac{1}{\sin^2(z)}}
	\end{equation}
	gives an expression for cosine weighted at its extrema. The result will be
	\begin{equation}
		\cos(z) = \frac{ \sum_{k=-\infty}^{\infty} \frac{(-1)^k}{(z-k\pi)^2} }{ \sum_{k=-\infty}^{\infty} \frac{1}{(z-k\pi)^2} } \, .
	\end{equation}
	Fig.~\ref{fig:cosine-interpolation-example} shows approximate interpolations resulting from finite versions of this sum.
	To demonstrate this result, we first find the principal parts of the numerator $\sfrac{\cos(z)}{\sin^2(z)}$ and denominator $\sfrac{1}{\sin^2(z)}$ of Eq.~(\ref{eq:cosine-barycentric-idea}), and then find their analytic remainders.
	
	\begin{figure}[h]
		\centering
		\includegraphics[width=0.75\linewidth]{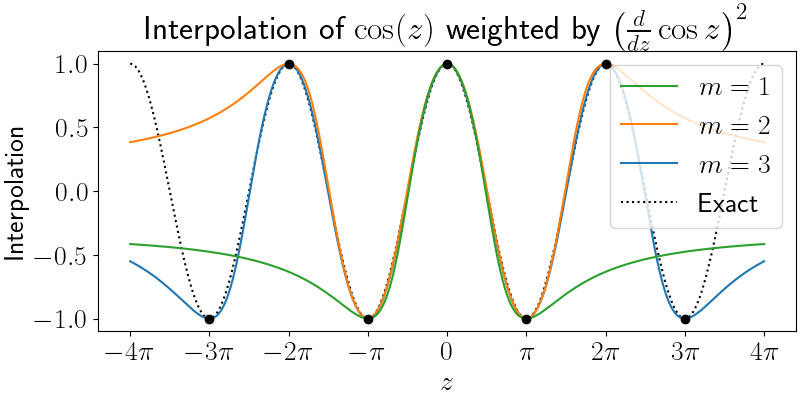}
		\caption{Here we plot an interpolation based off an expansion of $\cos z = \frac{\sfrac{\cos z}{\sin^2 z}}{\sfrac{1}{\sin^2 z}}$ and its evolution as more nodes are added. Each curve includes the node $z_0 = 0$ and the next $N$ nodes for a given $N$.}
		\label{fig:cosine-interpolation-example}
	\end{figure}
	
	The Laurent series principal parts are easy to obtain from the power series for the trigonometric functions,
	\begin{equation} \label{eq:reciprocal-trig-power-series}
		\csc(z) = \frac{1}{z} + \frac{z}{6} + \dots \, , \quad\quad
		\cos(z) = 1 - \frac{z^2}{2} + \dots \, ,
	\end{equation}
	and their antiperiodicity over $\pi$. We find that the Laurent series principal part of $\frac{\cos(z)}{\sin^2(z)}$ at $z_k = \pi k$ is $\frac{(-1)^n}{(z-k\pi)^2}$, and likewise the principal part of $\frac{1}{\sin^2(z)}$ at $z_k$ is $\frac{1}{(z-k\pi)^2}$.
	
	To evaluate the analytic remainders for the numerator $\frac{\cos z}{\sin^2 z}$ and denominator $\frac{1}{\sin^2 z}$, we take $\Omega_m$ to be a square centered at the origin with its edges at $x = \pm R_m$, $y = \pm R_m$ with $R_m = \pi\left(m+\frac{1}{2}\right)$, and we extend Eq.~(\ref{eq:I-integral-def}) by adding an index to label the numerator or denominator,
	\begin{equation} \label{eq:I-integral-extension}
		I_{(S,m)}^{(num,j)} = \int_{\Gamma_m^{(S)}} dw \frac{1}{w^{j+1}} \frac{\cos w}{\sin^2 w} \, , \quad
		I_{(S,m)}^{(den,j)} = \int_{\Gamma_m^{(S)}} dw \frac{1}{w^{j+1}} \frac{1}{\sin^2 w} \, .
	\end{equation}
	By the parity argument of Prop.~\ref{prop:analytic-remainder-parity-argument}, it suffices to calculate these integrals for only the right and top sides, $S=R,T$. We will show that $I_{(S,m)}^{(num,j)}$ and $I_{(S,m)}^{(den,j)}$ vanish for both these sides, and hence the integral along the entire contour vanishes, for all $j$.
	
	Using $R_m = \left(m+\frac{1}{2}\right)\pi$ and regular and hyperbolic trigonometric identities, we can show that
	\begin{equation}\begin{aligned} \label{eq:cosine-inequalities}
		\left|\cos\left(R_m + i y\right)\right| = \left|\sinh y\right| < e^{|y|} \, ,
		&\quad \left|\sin\left(R_m + i y\right)\right|^2 = \cosh^2 y > \frac{1}{4} e^{2|y|} \, , \\
		\left|\cos\left(x + i R_m\right)\right| \le \cosh R_m \, ,
		& \quad \left|\sin\left(x + i R_m\right)\right|^2 \ge \sinh^2 R_m  \, .
	\end{aligned}\end{equation}
	We must now evaluate the integrals of Eq.~(\ref{eq:I-integral-extension}) in four cases, the right and top sides of the contour for each of the numerator $\frac{\cos z}{\sin^2 z}$ and denominator $\frac{1}{\sin^2 z}$. Using the triangle inequality,
	\begin{align}
		\left|I_{(R,m)}^{(num,j)}\right|
		& = \left|\int_{\Gamma_m^{(R)}} dw \frac{1}{w^{j+1}} \frac{\cos w}{\sin^2 w}\right|
		< \frac{1}{R_m^{j+1}} \int_{-R_m}^{R_m} dy \frac{e^{|y|}}{\frac{1}{4}e^{2|y|}}
		< \frac{8}{R_m^{j+1}} \, , \\
		\left|I_{(R,m)}^{(den,j)}\right|
		& = \left|\int_{\Gamma_m^{(R)}} dw \frac{1}{w^{j+1}} \frac{1}{\sin^2 w}\right|
		< \frac{1}{R_m^{j+1}} \int_{-R_m}^{R_m} dy \frac{1}{\frac{1}{4}e^{2|y|}}
		< \frac{4}{R_m^{j+1}} \, , \\
		\left|I_{(T,m)}^{(num,j)}\right|
		& = \left|\int_{\Gamma_m^{(T)}} \frac{1}{w^{j+1}} dw \frac{\cos w}{\sin^2 w}\right|
		\le \frac{2}{R_m^j} \frac{\cosh R_m}{\sinh^2 R_m} \, , \label{eq:cos-2-num-top} \\
		\left|I_{(T,m)}^{(den,j)}\right|
		& = \left|\int_{\Gamma_m^{(T)}} \frac{1}{w^{j+1}} dw \frac{1}{\sin^2 w}\right|
		\le \frac{2}{R_m^j} \frac{1}{\sinh^2 R_m} \, , \label{eq:cos-2-den-top}
	\end{align}
	Each of these vanishes in the limit $m\to\infty$ regardless of $j$. Therefore the integral of the entire contour vanishes in the limit
	\begin{equation}
		\lim_{m \to \infty} \oint_{\partial\Omega_m} dw \frac{1}{w^{j+1}} \frac{\cos w}{\sin^2 w} = \lim_{m \to \infty} \oint_{\partial\Omega_m} dw \frac{1}{w^{j+1}} \frac{1}{\sin^2 w} = 0
	\end{equation}
	for all $j$, so that the analytic remainders for both the numerator and the denominator vanish and only the principal parts remain,
	\begin{equation}\begin{aligned} \label{eq:cosine-barycentric-proof}
		\cos(z) & = \lim_{m \to \infty} \frac{ \sum_{k=-m}^{m} \frac{(-1)^k}{(z-k\pi)^2} + \sum_{j=0}^\infty z^j \left( \frac{1}{2 \pi i} \oint_{\partial\Omega_m} dw \frac{1}{w^{j+1}} \frac{\cos(z)}{\sin^2 (z)} \right)}{ \sum_{k=-m}^{m} \frac{1}{(z-k\pi)^2} + \sum_{j=0}^\infty z^j \left( \frac{1}{2 \pi i} \oint_{\partial\Omega_m} dw \frac{1}{w^{j+1}} \frac{1}{\sin^2 (z)} \right)} \\
		& = \frac{ \sum_{k=-\infty}^{\infty} \frac{(-1)^k}{(z-k\pi)^2} }{ \sum_{k=-\infty}^{\infty} \frac{1}{(z-k\pi)^2} } \, .
	\end{aligned}\end{equation}
	
	If we were to take only a finite sum in the first line of Eq.~\ref{eq:cosine-barycentric-proof}, then according to Prop.~\ref{prop:F-error-convergence-rate}, the convergence of the finite approximation would be given by the most slowly converging of any of the analytic remainder integrals. In this case, both the numerator and denominator converge as $R_m^{-1}$. We will illustrate the convergence of this interpolation formula alongside other possible interpolation formulae for $\cos(z)$ in the next section.
		
	\subsection{Additional barycentric expansions of cosine} \label{sec:different-sinusoids}
	
	To construct different barycentric interpolation formulae for the same function, we choose a different weight function, and there are two natural choices to explore next. First, we might adjust the power of the weight function, which will fix higher-order derivatives at each node as showed in Prop.~\ref{prop:anchor-point-fixed-derivative}. Second, we might choose a different set of nodes.
	
	In the following sections, we will explore each of these possibilities and compare the convergence of finite approximations of these interpolations as more nodes are included.
	The calculations for their Laurent series principal parts and analytic remainders are essentially identical to those we have already performed, so we will omit them.
	
	\subsubsection{Varying powers of the weight function} \label{sec:cos-denom-power-weights}
	
	If we modify Eq.~(\ref{eq:cosine-barycentric-idea}) to
	\begin{equation} \label{eq:weighting-scheme-denom}
		F^p (z) = \frac{ \sfrac{\cos(z)}{\left(\frac{d}{dz}\cos(z)\right)^p} }{ \sfrac{1}{\left(\frac{d}{dz}\cos(z)\right)^p} } \, ,
	\end{equation}
	we obtain the following three expansions.
	\begin{subequations}\begin{align}
		\cos z & = \frac{\sfrac{\cos z}{\sin z}}{\sfrac{1}{\sin z}} = \frac{\sum_{k=-\infty}^{\infty} \frac{1}{z-k\pi} }{ \sum_{k=-\infty}^{\infty} \frac{(-1)^k}{z-k\pi} } \, , & p = 1 \, , \\
		\cos z & = \frac{\sfrac{\cos z}{\sin^2 z}}{\sfrac{1}{\sin^2 z}} = \frac{ \sum_{k=-\infty}^{\infty} \frac{(-1)^k}{(z-k\pi)^2} }{ \sum_{k=-\infty}^{\infty} \frac{1}{(z-k\pi)^2} } \, , & p = 2 \, , \\
		\cos z & = \frac{\sfrac{\cos z}{\sin^3 z}}{\sfrac{1}{\sin^3 z}} = \frac{\sum_{k=-\infty}^{\infty} \left( \frac{1}{(z-k\pi)^3} - \frac{1}{15}\frac{1}{z-k\pi} \right)}{ \sum_{k=-\infty}^{\infty} \left( \frac{(-1)^k}{(z-k\pi)^3} + \frac{1}{2}\frac{(-1)^k}{z-k\pi} \right)} \, , \quad\quad & p = 3 \, .
	\end{align}\end{subequations}
	We have restated the $p=2$ case for completeness, and we have disregarded the negative sign from $\frac{d}{dz} \cos z = - \sin z$ for $p=1,3$ since they always cancel between the numerator and denominator.
	
	%\begin{figure}
	%	\centering
	%	% The subfigures have 3:1 and 1:1 aspect ratios; the sizing below prints them with 5% textwidth spacing between them.
	%	\begin{subfigure}[t]{0.7125\textwidth}
	%		\includegraphics[height=0.333333\linewidth]{img/cosine_with_p_comparison.png}
	%		\caption{Absolute error with $z$ in the interpolation interval}
	%	\end{subfigure}\hfill%
	%	\begin{subfigure}[t]{0.2375\textwidth}
	%		\includegraphics[height=\linewidth]{img/cosine_with_p_convergence.png}
	%		\caption{Square root of mean square error}
	%	\end{subfigure}
	%	\caption{Here we compare the absolute error for approximations of $F^p (z) = \frac{\sfrac{\cos z}{\sin^p z}}{\sfrac{1}{\sin^p z}}$ for $p=1,2,3$. Each approximation uses only the nodes $z_k$ for $-m \le k \le m$, i.e., from $-\pi m$ to $\pi m$.}
	%	\label{fig:cosine_den_pow_comparison}
	%\end{figure}
	\begin{figure}[h]
		\centering
		% The subfigures have 3:1 and 1:1 aspect ratios; the sizing below prints them with 5% textwidth spacing between them.
		\begin{subfigure}[t]{\textwidth}
			\centering
			\includegraphics[width=\linewidth]{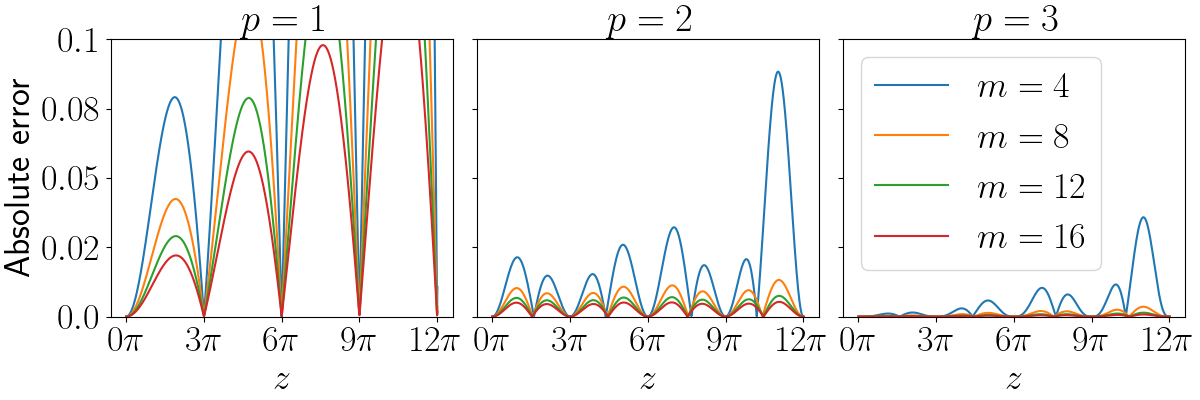}
			\caption{Absolute error with $z$ in the interpolation interval}
		\end{subfigure}%
		\vspace{0.025\linewidth}
		\begin{subfigure}[t]{\textwidth}
			\centering
			\includegraphics[width=0.6\linewidth]{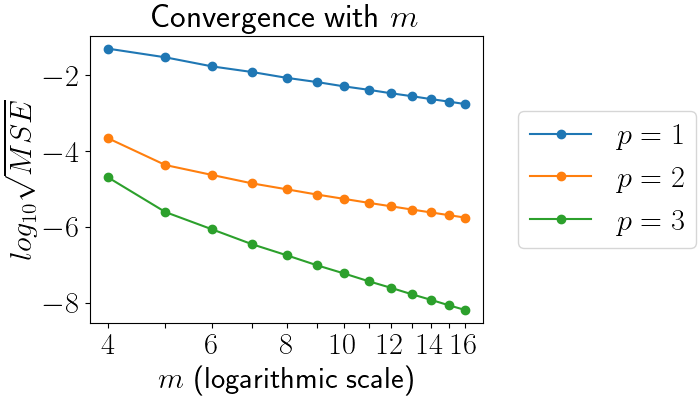}
			\caption{Square root of mean square error}
		\end{subfigure}
		\caption{Here we compare the absolute error and mean square error integrated over the same interpolation interval for approximations of $F^p (z) = \frac{\sfrac{\cos z}{\sin^p z}}{\sfrac{1}{\sin^p z}}$ for $p=1,2,3$. Each approximation uses only the nodes $z_k$ for $-m \le k \le m$, i.e., from $-\pi m$ to $\pi m$.}
		\label{fig:cosine_den_pow_comparison}
	\end{figure}
	
	As $p$ increases and raises the order of the pole at each node, the accuracy of the expansion increases in exchange for the expression becoming more complicated, with a longer Laurent principal part.
	Fig.~\ref{fig:cosine_den_pow_comparison} compares the accuracy of these three with $z$ on the real axis and shows their convergence as more nodes are added.
	These plots demonstrates not only a decrease in overall error with increasing $p$, which is due to the error-suppression effect discussed in Remark~\ref{remark:G-power-error-suppression}, but also a change of error scaling with $m$. $p=1$ and $p=2$ converge as $R_m^{-1}$, whereas $p=3$ converges as $R_m^{-2}$. (Since $R_m \propto m$, the scaling for $m$ is equivalent to the scaling for $R_m$.)
	
	As Prop.~\ref{prop:F-error-convergence-rate} shows, the overall error follows the scaling of the error in the most slowly-converging analytic remainder integral. For brevity, we will only focus on the most interesting case, $p=3$, for which the four analytic remainder integrals we must estimate are
	\begin{equation}\begin{aligned}
			I_{(S,m)}^{(num,j)} = \int_{\Gamma_m^{(S)}} dw \frac{1}{w^{j+1}} \frac{\cos z}{\sin^3 z} \, , \quad\quad I_{(S,m)}^{(den,j)} = \int_{\Gamma_m^{(S)}} dw \frac{1}{w^{j+1}} \frac{1}{\sin^3 z}
	\end{aligned}\end{equation}
	for each side $S$.
	Both $\sfrac{\cos z}{\sin^3 z}$ and $\sfrac{1}{\sin^3 z}$ decay quickly enough to $0$ as $|z|\to\infty$ that in both cases $\lim_{m\to\infty} \left| \int_{\Gamma_m^{(S)}} f(w) \right|$ is a constant. Therefore, we can isolate the factor of $\sfrac{1}{w^{j+1}}$, which sets the scaling:
	\begin{equation}
		\left|I_{(S,m)}^{(j)}\right|
		= \left| \int_{\Gamma_m^{(S)}} \frac{f(w)}{w^{j+1}} \right|
		< \frac{1}{R_m^{j+1}} \left| \int_{\Gamma_m^{(S)}} f(w) \right|
		= O\left(R_m^{-(j+1)}\right)
	\end{equation}
	The $j=0$ terms vanish and and do not contribute by the parity argument of Prop.~\ref{prop:analytic-remainder-parity-argument} since $0$ is even and both $\sfrac{\cos z}{\sin^3 z}$ and $\sfrac{1}{\sin^3 z}$ are odd.
	The lowest-order integrals that do not vanish are the $j=1$ terms, which converge as $R_m^{-2}$, as claimed.
	
	We have focused our attention inside the interpolation region. It is also worth mentioning the behavior of each approximant outside the interpolation region as well. The limiting behavior is set by the lowest-order nonvanishing terms in the Laurent series principal parts, and the present cases have the limiting behavior
	\begin{subequations} \label{eq:cos-den-limiting-behavior} \begin{align}
		\lim_{|z|\to\infty} F_m^1 (z) & = (-1)^m (2m+1) \, , \\
		\lim_{|z|\to\infty} F_m^2 (z) & = \frac{(-1)^m}{2m+1} \, , \\
		\lim_{|z|\to\infty} F_m^3 (z) & = (-1)^{m+1} \frac{2}{15} (2m+1) \, .
	\end{align}\end{subequations}
	These show that the large-$|z|$ limit is not necessarily stable as $m$ grows. Where global accuracy including outside the interpolation region is relevant, this also should be considered when choosing a weight function.
	
	\subsubsection{Weighting at different nodes} \label{sec:cos-denom-not-extrema}
	
	So far we have only explored barycentric expansions resulting from anchoring at extrema. Here we also consider anchoring at zeros, at both extrema and zeros together, and at the points at midpoints between extrema and zeros, all for both $p=1$ and $p=2$. The expressions are not particularly illuminating, so we defer them to Appendix~\ref{sec:cosine-different-anchors} and only study their properties here.
	
	\begin{figure}[b!]
		\centering
		\begin{subfigure}{\textwidth}
			\centering
			\includegraphics[width=\linewidth]{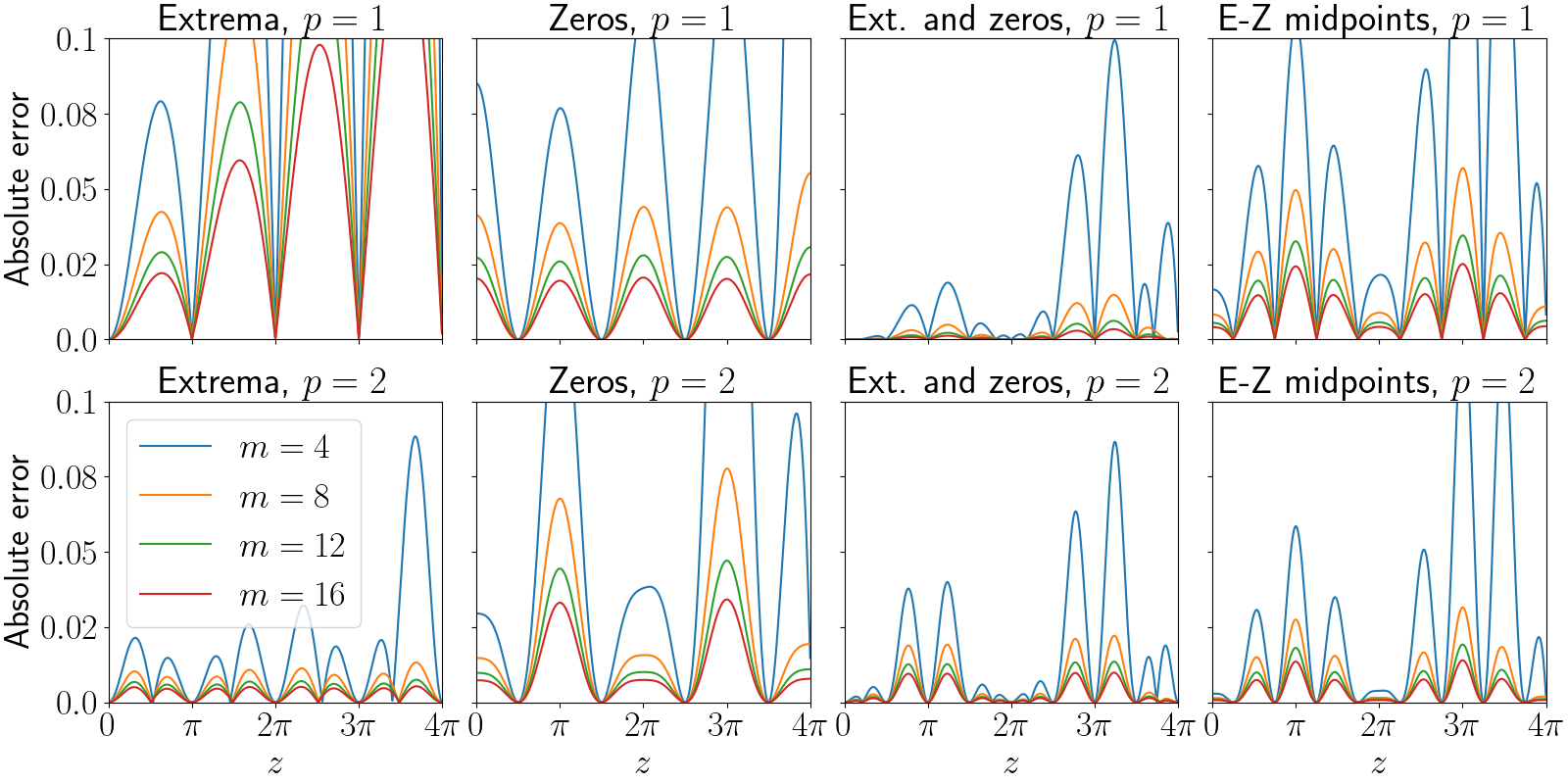}
			\caption{Absolute error as a function $z$ inside the interpolation interval}
		\end{subfigure}%
		\vspace{0.025\linewidth}
		\begin{subfigure}{\textwidth}
			\centering
			\includegraphics[width=\linewidth]{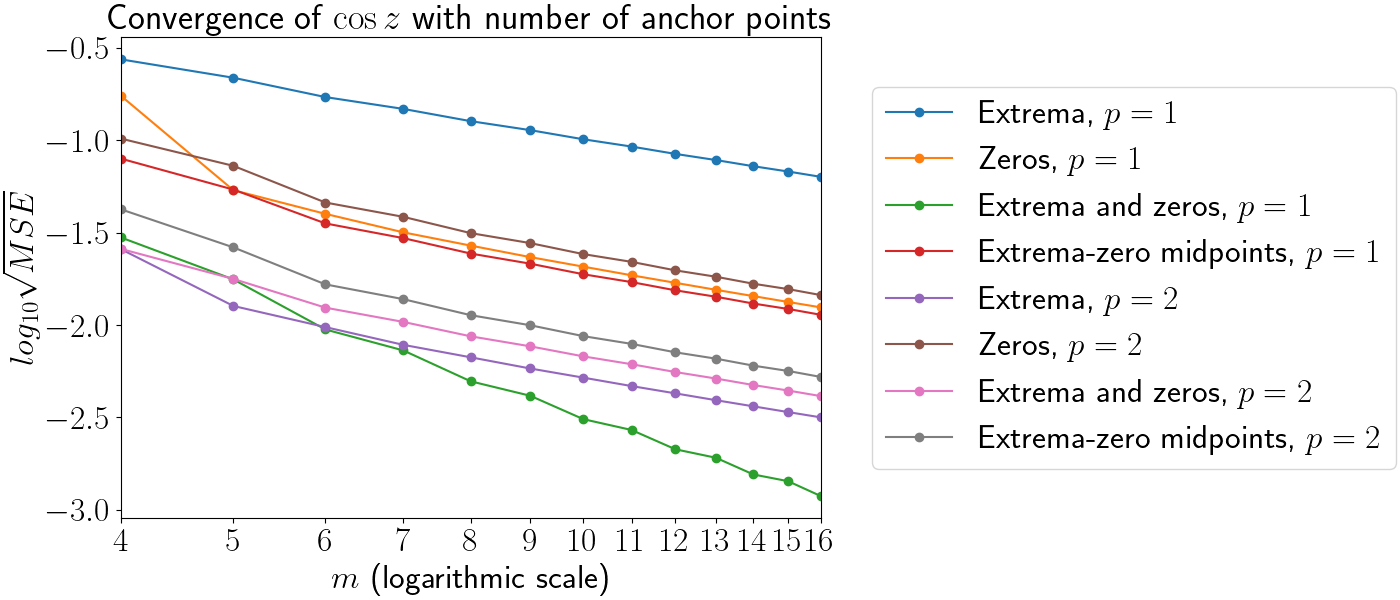}
			\caption{Square root of mean square error over the same interval}
		\end{subfigure}
		\caption{Here we compare interpolations of $\cos(z)$ for different choices of nodes and for differing numbers of nodes included. These are indexed by $m$ such that the interpolation includes all nodes between $-\pi m$ and $\pi m$. As before, $p$ is the power to which the weight function is raised.}
	\label{fig:cosine_anchor_points_comparison}
	\end{figure}
	
	Fig.~\ref{fig:cosine_anchor_points_comparison} compares the accuracy of these inside the interpolation interval as the number of nodes varies. Convergence occurs quickly enough that the interpolation formula with the best scaling is the most accurate unless relatively few nodes are used.
	The $p=1$ interpolation weighted at extrema and zeros has better scaling than the other expressions for the same reason as did the $p=3$ interpolation weighted at extrema in the preceding section.
	Namely, both $\sfrac{\cos z}{\sin 2z}$ and $\sfrac{1}{\sin 2z}$ are odd and tend to $0$ quickly for large $z$. Therefore the $j$th term in the error scales as $R_m^{-(j+1)}$, but $j=0$ term vanishes so that the lowest-order term is $R_m^{-2}$. The $p=1$ interpolation anchored at extrema and zeros is the only one of these expressions that satisfies these properties and hence is the only one that shows this scaling.
	This example illustrates that symmetry (or antisymmetry) is an important factor in producing efficient interpolations. Another important factor which we have not explored here is the spacing of nodes, which affects how $R_m$ scales with $m$.
	
	While Remark~\ref{remark:G-power-error-suppression} suggests that raising the weight function to a higher power should reduce error, the $p=1$ and $p=2$ interpolations weighted at zeros show an exception, even though both have the same error scaling as $R_m^{-1}$.
	The reason is that the numerator for the $p=1$ case is just $f^{(num)}(z) = \sfrac{\cos z}{\cos z} = 1$. There are clearly no poles here to expand into a sum over Laurent series principal parts. However, truncating the sums over Laurent series principal parts ie exactly how we approximate $f^{(num)}(z)$ by $f_m^{(num)}(z)$ in the first place in order to construct approximate barycentric interpolations using Eq.~(\ref{eq:gen-bary-rat-interp-approx}). If there are no terms to truncate anyway, then the "approximate" expansion is always exact, and the numerator does not contribute anything to the error of the overall interpolation. That is, $e_m^{(num)}(z) = 0$ always in Eq.~(\ref{eq:F-error-with-e}). Here we have focused only on analytical and numerical properties; whether $\cos z = \frac{\sfrac{\cos z}{\cos z}}{\sfrac{1}{\cos z}} = \frac{1}{\sfrac{1}{\cos z}}$ can truly be called "barycentric" with no weight-like terms in the numerator is a separate question.
	
	\section{Barycentric rational expansions for the Bessel functions} \label{sec:bessel-barycentric}
	
	Now let us generate barycentric rational expansions for the more novel and interesting case of the Bessel functions of the first kind, $J_q (z)$, for integer order $q$. Based on our observations for $\cos z$ in Sec.~\ref{sec:different-sinusoids}, we will derive and compare interpolations weighted at extrema, zeros, or both extrema and zeros for $p=1,2$.
	
	As we will show in Sec.~\ref{sec:bessel-results}, the best-performing interpolation is again that weighted at extrema with $p=2$. For $q=0$, the result has the form
	\begin{equation}
		J_0(z) = \frac{\frac{4}{z^2} + \sum_{k=-\infty, k\ne0}^\infty \left(\frac{a_{-2}^{(k,num)}}{\left(z-z_k\right)^2} + \frac{a_{-1}^{(k,num)}}{z-z_k}\right)}
		{1 + \frac{4}{z^2} + \sum_{k=-\infty,k\ne0}^\infty \left(\frac{a_{-2}^{(k,den)}}{\left(z-z_k\right)^2} + \frac{a_{-1}^{(k,den)}}{z-z_k}\right) } \, .
	\end{equation}
	Fig.~\ref{fig:bessel-interpolation-example} shows approximate interpolations resulting from finite  versions of this expansion.
	The values of the Laurent coefficients and the expansions for other orders $q$ are given in Sec.~\ref{sec:bessel-results}, and expansions for other weightings are given in Appendix~\ref{sec:bessel-expansions-calculation}.
	
	\begin{figure}
		\centering
		\includegraphics[width=0.75\linewidth]{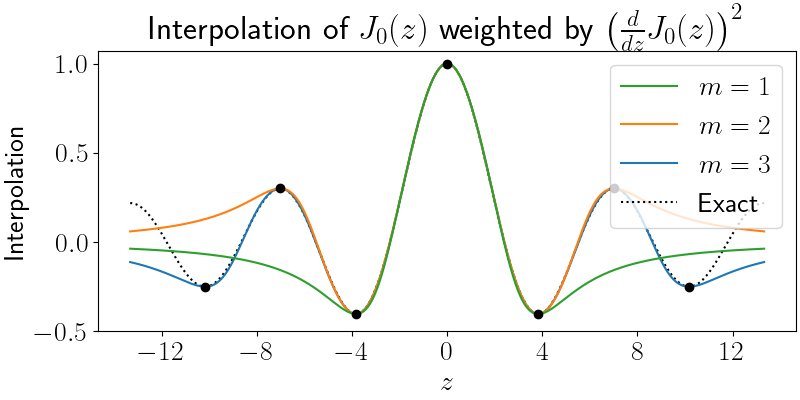}
		\caption{Here we plot an interpolation based off an expansion of $J_0(z) = \frac{\sfrac{J_0(z)}{\left(\frac{d}{dz}J_0(z)\right)^2}}{\sfrac{1}{\left(\frac{d}{dz}J_0(z)\right)^2}}$ and its evolution as more nodes are added. Each curve includes the node $z_0 = 0$ and the next $N$ nodes for a given $N$.}
		\label{fig:bessel-interpolation-example}
	\end{figure}
	
	The first step in producing the barycentric expansions is calculating Laurent series principal parts. Unlike the poles in the various cosine expansions in Sec.~\ref{sec:cos-barycentric}, for the Bessel functions, not all poles are of the same order. The order of the zero of $J_q(z)$ at $z=0$ depends on the index $q$, and hence so does the order of the pole at $z_0=0$. The zero and pole orders for each choice of weight function are given in Table~\ref{tbl:bessel-pole-orders-at-0}.
	
	\begin{table}
		\centering
		\begin{tabular}{c | c | c | c}
			\multicolumn{4}{c}{Orders of zeros and poles at $z=0$} \\
			\hline\hline
			& $q=0$ & $q=1$ & $q\ge2$ \\
			\hline\hline
			$J_q(z)$ zero order & $0$ & $1$ & $q$ \\
			\hline
			$\frac{d}{dz}J_q(z)$ zero order & $1$ & $0$ & $q-1$ \\
			\hline\hline
			$\sfrac{J_q (z)}{\frac{d}{dz}J_q (z)}$ pole order & $1$ & $0$ & $0$ \\
			\hline
			$\sfrac{1}{\frac{d}{dz}J_q (z)}$ pole order & $1$ & $0$ & $q-1$ \\
			\hline\hline
			$\sfrac{J_q (z)}{\left(\frac{d}{dz}J_q (z)\right)^2}$ pole order & $2$ & $0$ & $q-2$ \\
			\hline
			$\sfrac{1}{\left(\frac{d}{dz}J_q (z)\right)^2}$ pole order & $2$ & $0$ & $2q-2$ \\
			\hline\hline
			$\sfrac{1}{J_q (z)}$ pole order & $0$ & $1$ & $q$ \\
			\hline
			$\sfrac{1}{\left(J_q (z)\right)^2}$ pole order & $0$ & $2$ & $2q$ \\
			\hline\hline
			$\sfrac{1}{J_q (z)\frac{d}{dz}J_q (z)}$ pole order & $1$ & $1$ & $2q-1$
		\end{tabular}
		\caption{Whether or not $z=0$ is a pole of these functions depends on the order $q$ of the Bessel function. Here, "order $0$" for a zero or pole means the absence of the zero or pole.}
		\label{tbl:bessel-pole-orders-at-0}
	\end{table}
	
	Second, we must evaluate the analytic remainder integrals. We again choose a sequence of easy-to-integrate square regions $\Omega_m$ centered on the origin with sides at $x,y = \pm R_m$, this time with
	\begin{equation}\begin{aligned} \label{eq:bessel-Rm}
		R_m = \tfrac{\pi}{2}\left(2m+q+\tfrac{3}{2}\right) , \quad\quad & \text{when weighting at extrema} , \\
		R_m = \tfrac{\pi}{2}\left(2m+q+\tfrac{5}{2}\right) , \quad\quad & \text{when weighting at zeros} , \\
		R_m = \tfrac{\pi}{2}\left(m+q+1\right) , \quad\quad & \text{when weighting at both} .
	\end{aligned}\end{equation}
	For sufficiently large $m$, with these choices $R_m$ is roughly the midpoint between successive extrema or zeros, respectively, of $J_q(z)$. This follows from the asymptotic expansions
	\begin{equation} \label{eq:bessel-asymptotic}
		\begin{array}{rr}
			& J_q(z) = \sqrt{\frac{2}{\pi z}}\left(\cos\left(z \! - \! \frac{\pi}{2}q \! - \! \frac{\pi}{4}\right)+e^{\left|\text{Im}(z)\right|}O\left(|z|^{-1}\right)\right) \\
			& \frac{d}{dz} J_q(z) = - \sqrt{\frac{2}{\pi z}}\left(\sin\left(z \! - \! \frac{\pi}{2}q \! - \! \frac{\pi}{4}\right)+e^{\left|\text{Im}(z)\right|}O\left(|z|^{-1}\right)\right)
		\end{array}
		\,\, |\arg(z)|<\pi \, ,
	\end{equation}
	which we will use in evaluating the contour integrals for the analytic remainder.\cite{noauthor_nist_2025} (Since $R_m \le |z| \le \sqrt{2} R_m$ for all $z \in \partial\Omega_m$, we can always replace $O\left(|z|^\alpha\right)$ in the error term with $O\left(R_m^\alpha\right)$.) The latter follows from the former when we use the derivative recurrence relation
	\begin{equation} \label{eq:bessel-recurrence}
		\frac{d}{dz} J_q (z) = \frac{1}{2}\left(J_{q-1}(z)-J_{q+1}(z)\right)\, .
	\end{equation}
	Finally, due to the parity of the Bessel functions for integer order $q$,
	\begin{equation} \label{eq:bessel-parity}
		J_q (-z) = (-1)^q J_q (z) \, ,
	\end{equation}
	we can use (anti)symmetry to make Eq.~(\ref{eq:bessel-asymptotic}) useful regardless of $\arg(z)$. This parity relation will also allow us to make use of the parity argument of Prop.~\ref{prop:analytic-remainder-parity-argument} to simplify the evaluation of certain analytic remainder integrals.
	
	Calculating the Laurent series principal parts and analytic remainders is more involved for the Bessel functions than for cosine, se we outline the procedure in the following sections and defer the details to Appendix~\ref{sec:bessel-expansions-calculation}.
	
	\subsection{Calculating the Laurent series principal parts} \label{sec:bessel-laurent}
	
	First we calculate the principal parts of the Laurent expansions. In general, one can find the the Laurent principal part coefficients $a_{-n}^{(k)}$ of a ratio of analytic functions of the form $\sfrac{c(z)}{b(z)}$ (e.g., $\sfrac{J_q(z)}{\frac{d}{dz}J_q(z)}$) by expanding them in power series about each node $z_k$ and matching like terms. To simplify notation, in the following, we will suppress the node index $k$ in all coefficients and shift $z-z_k \to z$ so that the expansion is in powers of $z$.
	
	Letting $j_0$ and $j_*$ be the respective orders of the zeros of $c(z)$ and $b(z)$ at $z=z_k$ and expanding the functions in power series with coefficients $c_j$ and $b_j$, we obtain a matrix equation for the Laurent principal part coefficients $a_{-j}$ as follows:
	\begin{align} \label{eq:laurent-coeff-matrix-equation-first-order-denom}
		& \frac{\sum_{j=j_0}^\infty c_j z^j}{\sum_{j=j_*}^\infty b_j z^j} = \frac{c(z)}{b(z)} = \sum_{j=1}^{j_*-j_0} \frac{a_{-j}}{z^j} + O(1) \nonumber \\
		& \implies \quad \sum_{j=j_0}^\infty c_j z^j = \sum_{j=j_*}^\infty \sum_{j'=1}^{j_*-j_0} a_{-j'} b_j z^{j-j'} + O\left(z^{j_*}\right) \\
		& \implies
		\begin{bmatrix}
			c_{j_0} \\ c_{j_0+1} \\ \ldots \\ c_{j_*-2} \\ c_{j_*-1}
		\end{bmatrix}
		=
		\begin{bmatrix}
			b_{j_*} & 0 & \ldots & 0 & 0 \\
			b_{j_*+1} & b_{j_*} & \ldots & 0 & 0 \\
			\ldots & \ldots & \ldots & \ldots & \ldots \\
			b_{2j_*-j_0-2} & b_{2j_*-j_0-3} & \ldots & b_{j_*} & 0 \\
			b_{2j_*-j_0-1} & b_{2j_*-j_0-2} & \ldots & b_{j_*+1} & b_{j_*} \\
		\end{bmatrix}
		\begin{bmatrix}
			a_{-\left(j_*-j_0\right)} \\ a_{-\left(j_*-j_0-1\right)} \\ \ldots \\ a_{-2} \\ a_{-1}
		\end{bmatrix} \nonumber \, .
	\end{align}
	In the second step we multiplied through by the denominator and in the third step we matched powers of $z$ from $z^{j_0}$ to $z^{j_*-1}$. We can write this matrix equation succinctly as
	\begin{equation} \label{eq:laurent-coeff-mat-eq-simple}
		c = B a
	\end{equation}
	where the vector and matrix elements $[a]_j$, $[B]_{j,j'}$, $[c]_j$ are given by
	\begin{align}
		\begin{matrix*}
			\left[a\right]_j & = & a_{-(j_* - j_0) + j} \, , \\
			\left[B\right]_{j,j'} & = & b_{j_* + j - j'} \Theta\left(j-j'\right) \, , \\
			\left[c\right]_j & = & c_{j_0 + j} \, ,
		\end{matrix*}
		\left.\vphantom{\begin{matrix*}
				\left[a\right]_j & = & a_{-(j_* - j_0) + j} \, , \\
				\left[B\right]_{j,j'} & = & b_{j_* + j - j'} \Theta\left(j-j'\right) \, , \\
				\left[c\right]_j & = & c_{j_0 + j} \, ,
		\end{matrix*}}\right\}
		\quad 0 \le j,j' < j_* - j_0 - 1
	\end{align}
	Note that we use the convention that the vector and matrix indices start at $0$ rather than $1$, and that brackets distinguish the matrix elements from the power series coefficients.
	The power series coefficients $c_j$ and $b_j$ are known, so the last step is simply to solve the matrix equation for the unknown vector of Laurent principal part coefficients $a$.
	
	The formulation above is best suited to the $p=1$ cases for extrema or zeros. For the remaining cases, when the denominator is a product of simpler functions, we obtain in the same way
	\begin{equation}\begin{aligned} \label{eq:laurent-coeff-matrix-equation-second-order-denom}
		& \frac{\sum_{j=j_0}^\infty c_j z^j}{\left(\sum_{j=j_{*0}}^\infty b_j^{(0)} z^j\right)\left(\sum_{j=j_{*1}}^\infty b_j^{(1)} z^j\right)}
		= \frac{c(z)}{b_0(z)b_1(z)}
		= \sum_{j=1}^{j_*-j_0} \frac{a_{-j}}{z^j} + O(1) \\
		& \quad\quad\quad \implies c = B_0 B_1 a
	\end{aligned}\end{equation}
	with  vector and matrix elements given by
	\begin{align} \label{eq:laurent-coeff-matrix-entries-second-order-denom}
		\begin{matrix*}
			\left[a\right]_j & = & a_{-(j_{*0} + j_{*1} - j_0) + j} \, , \\
			\left[B_d\right]_{j,j'} & = & b_{j_{*d} + j - j'}^{(d)} \Theta\left(j-j'\right) \, , \\
			\left[c\right]_j & = & c_{j_0 + j} \, ,
		\end{matrix*}
		\left.\vphantom{\begin{matrix*}
				\left[a\right]_j & = & a_{-(2j_* - j_0) + j} \, , \\
				\left[B_d\right]_{j,j'} & = & b_{j_{*d} + j - j'}^{(d)} \Theta\left(j-j'\right) \, , \\
				\left[c\right]_j & = & c_{j_0 + j} \, ,
		\end{matrix*}}\right\}
		\quad
		\begin{matrix*}
			0 \le j,j' < j_{*0} + j_{*1} - j_0 - 1 \\
			d = 0, 1
		\end{matrix*}
	\end{align}
	
	For most of the Bessel function expansions we present here, the dimension of the system of equations scales with $q$ at $z_0=0$ and is only $1$ or $2$ at all other nodes. Hence, the Laurent principal part coefficients can easily be obtained in closed form for all nodes $z_k$ except $z_0$, for which the coefficients are best found computationally.
	
	\subsection{Calculating the analytic remainders} \label{sec:bessel-analytic-remainders}
	
	As was the case for $\cos z$ in Sec.~\ref{sec:first-cos-expansion}, due to the symmetry of the Bessel functions as given in Eq.~(\ref{eq:bessel-parity}) we only need to calculate the right side and top side integrals. We can perform these integrals in a roughly similar way after expanding the Bessel functions using their asymptotic expansions. As concrete examples, we will calculate two integrals for weighting at the extrema, one with $p=1$ and one with $p=2$. The other integrals may be evaluated using similar methods to these two and we give them in Appendix~\ref{sec:bessel-expansions-calculation}.
	
	First, we calculate
	\begin{equation} \label{eq:bessel-example-integral-1}
		I_{(R,m)}^{(num,j)} = \int_{\Gamma_m^{(R)}} dw \frac{1}{w^{j+1}} \frac{J_q(w)}{\frac{d}{dw}J_q(w)} \, ,
	\end{equation}
	from the numerator for the expansion based on $\frac{\sfrac{J_q(z)}{\frac{d}{dz}J_q(z)}}{\sfrac{1}{\frac{d}{dz}J_q(z)}}$, weighted at extrema with $p=1$. With $R_m = \frac{\pi}{2}\left(2m+q+\frac{3}{2}\right)$ as chosen in Eq.~(\ref{eq:bessel-Rm}),
	\begin{equation}\begin{aligned} \label{eq:bessel-right-side-simplification-example}
			\sin\left(R_m+iy-\frac{\pi}{2}q-\frac{\pi}{4}\right)
			& = (-1)^m \cosh y \, , \\
			\cos\left(R_m+iy-\frac{\pi}{2}q-\frac{\pi}{4}\right)
			& = -i(-1)^m \sinh y \,.
	\end{aligned}\end{equation}
	Using these in the asymptotic forms of Eq.~(\ref{eq:bessel-asymptotic}) and expanding in a geometric series, we have
	\begin{align}
		\left.\frac{1}{\frac{d}{dw}J_q(w)}\right|_{R_m+iy}
		& = -\sqrt{\frac{\pi z}{2}} \frac{1}{(-1)^m \cosh y+e^{\left|y\right|}O\left(|R_m+iy|^{-1}\right)} \nonumber \\
		& = -\sqrt{\frac{\pi \left(R_m+iy\right)}{2}} \frac{(-1)^m}{\cosh y}\left(1 + O\left(R_m^{-1}\right)\right) \, , \label{eq:bessel-geometric-series-right-side-example} \\
		\left.\frac{J_q(w)}{\frac{d}{dw}J_q(w)}\right|_{R_m+iy}
		& = i \tanh y + O\left(R_m^{-1}\right) \, . \label{eq:bessel-geo-series-right-num-example}
	\end{align}
	In Eq.~(\ref{eq:bessel-geometric-series-right-side-example}), we used the asymptotic expansion of $\frac{d}{dz}J_q(z)$, and the validity of the geometric series follows for sufficiently large $R_m$ from $\cosh y \ge \frac{1}{2}e^{|y|}$.
	Eq.~(\ref{eq:bessel-geo-series-right-num-example}) then follows by multiplying Eq.~(\ref{eq:bessel-geometric-series-right-side-example}) the asymptotic expansion for $J_q(z)$. We have also consolidated all higher-order terms in the expansion with the lowest-order bound $O\left(R_m^{-1}\right)$ as this is sufficient for our needs.
	
	To evaluate the integral, we apply the ML inequality together with $\left|\tanh y\right| \le 1$ to see that
	\begin{equation}
		\left|I_{(R,m)}^{(num,j)}\right|
		= \left|\int_{\Gamma_m^{(R)}} dw \frac{1}{w^{j+1}} \frac{J_q(w)}{\frac{d}{dw}J_q(w)}\right|
		\le \frac{2}{R_m^j} + O\left(R_m^{-j-1}\right) \, , 
	\end{equation}
	The entire term vanishes for $j\ge1$ and the error term vanishes for all $j$. The $j=0$ case must be treated individually, but since here the integrand $\sfrac{J_q(w)}{\frac{d}{dw}J_q(w)}$ is odd while $j$ is even, by the parity argument of Prop.~\ref{prop:analytic-remainder-parity-argument}, $I_{(R,m)}^{(num,j)}$ cancels with $I_{(L,m)}^{(num,j)}$ so that this integral does not contribute to the analytic remainder.
	
	For the second example, we calculate
	\begin{equation} \label{eq:bessel-example-integral-2}
			I_{(S,m)}^{(den,j)} = \int_{\Gamma_m^{(S)}} dw \frac{1}{w^{j+1}} \frac{1}{\left(\frac{d}{dw}J_q(w)\right)^2} \, .
	\end{equation}
	from the denominator for the expansion based on $\frac{\sfrac{J_q(z)}{\left(\frac{d}{dz}J_q(z)\right)^2}}{\sfrac{1}{\left(\frac{d}{dz}J_q(z)\right)^2}}$, weighted at extrema with $p=2$.
	Squaring Eq.~(\ref{eq:bessel-geometric-series-right-side-example}), we have
	\begin{equation} \label{eq:bessel-geometric-series-right-side-squared-example}
		\frac{1}{\left(\frac{d}{dw}J_q(w)\right)^2}
		= \frac{\pi \left(R_m+iy\right)}{2} \frac{1}{\cosh^2 y}\left(1 + O\left(R_m^{-1}\right)\right) \, ,
	\end{equation}
	For $j\ge1$, the triangle inequality shows that Eq.~(\ref{eq:bessel-example-integral-2}) vanishes in the limit $m\to\infty$,
	\begin{align}
		\left|I_{(R,m)}^{(den,j)}\right|
		& = \left|\int_{\Gamma_m^{(R)}} dw \frac{1}{w^{j+1}} \frac{1}{\left(\frac{d}{dw}J_q(w)\right)^2}\right| \\
		& \le \frac{\pi}{2 R_m^j} \int_{-R_m}^{R_m} dy \frac{1}{\cosh^2 y} \left(1+O\left(R_m^{-1}\right)\right)
		= \pi R_m^{-j}\left(1+O\left(R_m^{-1}\right)\right) \, . \nonumber
	\end{align}
	However, in the case $j=0$ the integral does not vanish, instead giving
	\begin{equation}\begin{aligned}
			I_{(R,m)}^{(den,0)}
			& = \int_{\Gamma_m^{(R)}} dz \frac{1}{w} \frac{1}{\left(\frac{d}{dw}J_q(w)\right)^2} \\
			& = i \frac{\pi}{2} \int_{-R_m}^{R_m} dy \frac{1}{\cosh^2 y} \left(1+O\left(R_m^{-1}\right)\right) = i\pi + O\left(R_m^{-1}\right) \, .
	\end{aligned}\end{equation}
	This integral therefore yields a nonvanishing analytic remainder.
	
	The remaining integrals are evaluated in Appendix~\ref{sec:bessel-expansions-calculation} using the same techniques, namely, the asymptotic forms of Eq.~(\ref{eq:bessel-asymptotic}), the ML and triangle inequalities, and the parity argument of Prop.~\ref{prop:analytic-remainder-parity-argument}.
	
	\subsection{Comparison of Bessel function expansions} \label{sec:bessel-results}
	
	Following the example of Sec.~\ref{sec:cos-denom-not-extrema}, we derive expansions which use the extrema, zeros, or both extrema and zeros as nodes with both $p=1,2$, except for the case of weighting at both extrema and zeros with $p=2$, for which the expansion method presented here is not valid. The resulting expressions are given in Appendix~\ref{sec:bessel-expansions-calculation}; here we give here only the expansion weighted at extrema with $p=2$, which has the greatest accuracy and fastest convergence as shown in Figs.~\ref{fig:bessel-plot-comparison} and \ref{fig:bessel-plot-convergence}.
	
	\begin{figure}[hb!]
		\centering
		\includegraphics[width=\linewidth]{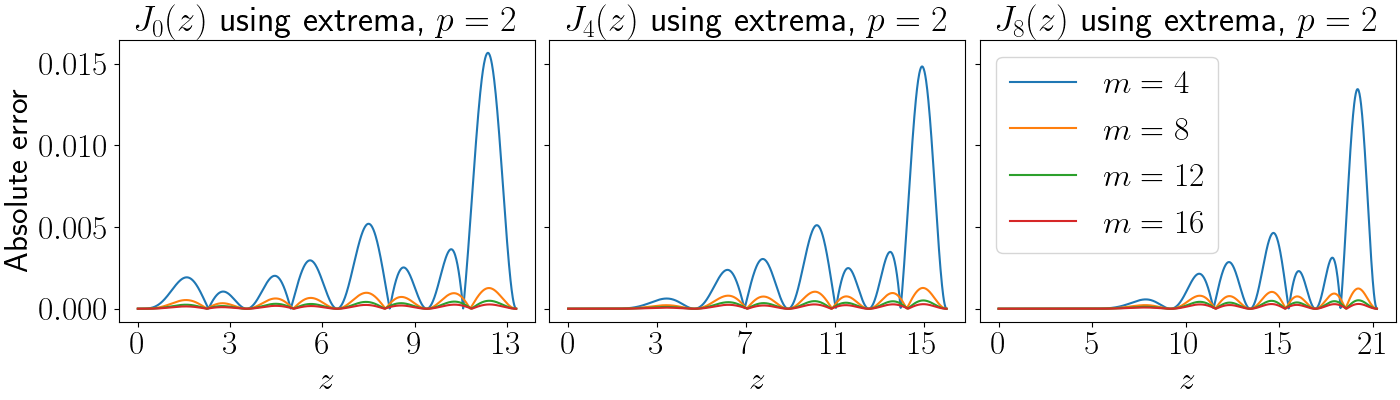}
		\caption{Here we plot the absolute error in the interpolation weighted at extrema with $p=2$ for varying order $q$ of $J_q(z)$. Since all nodes except $z_0=0$ shift away from the origin as $q$ increases, and the error profile of the approximation shifts with them, but its magnitude and shape remain essentially unchanged. For consistency we keep $q$ even and plot from $z_0$ to $z_4$ in each case.}
		\label{fig:bessel-order-comparison}
	\end{figure}
	
	The result for each order $q$ is
	\begin{equation}\begin{aligned} \label{eq:bessel-weighting-2-final-result-showcase}
		J_0(z)
		& = \frac{\frac{4}{z^2} + \sum_{k=-\infty, k\ne0}^\infty \left(\frac{a_{-2}^{(k,num)}}{\left(z-z_k\right)^2} + \frac{a_{-1}^{(k,num)}}{z-z_k}\right)}
		{1 + \frac{4}{z^2} + \sum_{k=-\infty,k\ne0}^\infty \left(\frac{a_{-2}^{(k,den)}}{\left(z-z_k\right)^2} + \frac{a_{-1}^{(k,den)}}{z-z_k}\right) } \, , \\
		J_1(z)
		& = \frac{\sum_{k=-\infty, k\ne0}^\infty \left(\frac{a_{-2}^{(k,num)}}{\left(z-z_k\right)^2} + \frac{a_{-1}^{(k,num)}}{z-z_k}\right)}
		{1 + \sum_{k=-\infty,k\ne0}^\infty \left(\frac{a_{-2}^{(k,den)}}{\left(z-z_k\right)^2} + \frac{a_{-1}^{(k,den)}}{z-z_k}\right) } \, , \\
		J_q(z)
		& = \frac{\sum_{n=1}^{q-2} \frac{a_{-n}^{(0,num)}}{z^n} + \sum_{k=-\infty, k\ne0}^\infty \left(\frac{a_{-2}^{(k,num)}}{\left(z-z_k\right)^2} + \frac{a_{-1}^{(k,num)}}{z-z_k}\right)}
		{1 + \sum_{n=1}^{2(q-1)} \frac{a_{-n}^{(0,den)}}{z^n} + \sum_{k=-\infty,k\ne0}^\infty \left(\frac{a_{-2}^{(k,den)}}{\left(z-z_k\right)^2} + \frac{a_{-1}^{(k,den)}}{z-z_k}\right) } \, .
	\end{aligned}\end{equation}
	The last case is for $q \ge 2$. The Laurent principal part coefficients $a_{-n}^{(k,num)}, a_{-n}^{(k,den)}$ for $k\ne0$ are given by
	\begin{equation}\begin{aligned} \label{eq:bessel-laurent-p2-example}
			a_{-2}^{(k,num)}
			& = \frac{J_q(z_k)}{(J_q^{(2)}(z_k))^2} \, , \quad
			a_{-1}^{(k,num)}
			= \frac{J_q^{(1)}(z_k)}{(J_q^{(2)}(z_k))^2} - \frac{(J_q(z_k))(J_q^{(3)}(z_k))}{(J_q^{(2)}(z_k))^3} \, , \\
			a_{-2}^{(k,den)}
			& = \frac{1}{(J_q^{(2)}(z_k))^2} \, , \quad
			a_{-1}^{(k,den)}
			= - \frac{J_q^{(3)}(z_k)}{(J_q^{(2)}(z_k))^3} \, .
	\end{aligned}\end{equation}
	For $k=0$ and $q\ge2$, the numerator Laurent principle part coefficients $a_{-n}^{(0,num)}$ are given by solving the matrix equation Eq.~(\ref{eq:laurent-coeff-matrix-equation-second-order-denom}) using the coefficients in Eq.~(\ref{eq:laurent-coeff-matrix-entries-second-order-denom}) with $j_0=q$ and $c_{q+2j} = \frac{(-1)^j}{2^{q+2j}} \frac{1}{j!(q+j)!}$ for $j\ge0$, while the denominator Laurent principal part coefficients $a_{-n}^{(0,den)}$ are given by solving the same matrix equation with $j_0=0$ and $c_j = \delta_{j,0}$. In both cases, $j_{*0}=j_{*1}=q-1$ and $b_{q-1+2j}^{(0)} = b_{q-1+2j}^{(1)} = \frac{(-1)^j}{2^{q+2j}} \frac{q+2j}{j!(q+j)!}$ for $j\ge0$. All $b$ and $c$ coefficients not covered by these indices vanish.
	
	The other Bessel function expansions besides that given in Eq.~(\ref{eq:bessel-weighting-2-final-result-showcase}) have similar forms; now let us compare them in greater detail. The accuracy of each expansion is essentially independent of the magnitude of $q$, as shown in Fig~\ref{fig:bessel-order-comparison}, but does depend on the parity of $q$ due to the vanishing of some analytic remainder integrals through Prop.~\ref{prop:analytic-remainder-parity-argument}. Hence, it is sufficient to examine only $q=0,1$.
	Fig.~\ref{fig:bessel-plot-comparison} shows the absolute error for these cases over an interval near the origin, and Fig.~\ref{fig:bessel-plot-convergence} shows how the total error over the same interval evolves with the number of nodes.
	
	\begin{figure}[h!]
		\centering
		\begin{subfigure}{\textwidth}
			\centering
			\includegraphics[width=\linewidth]{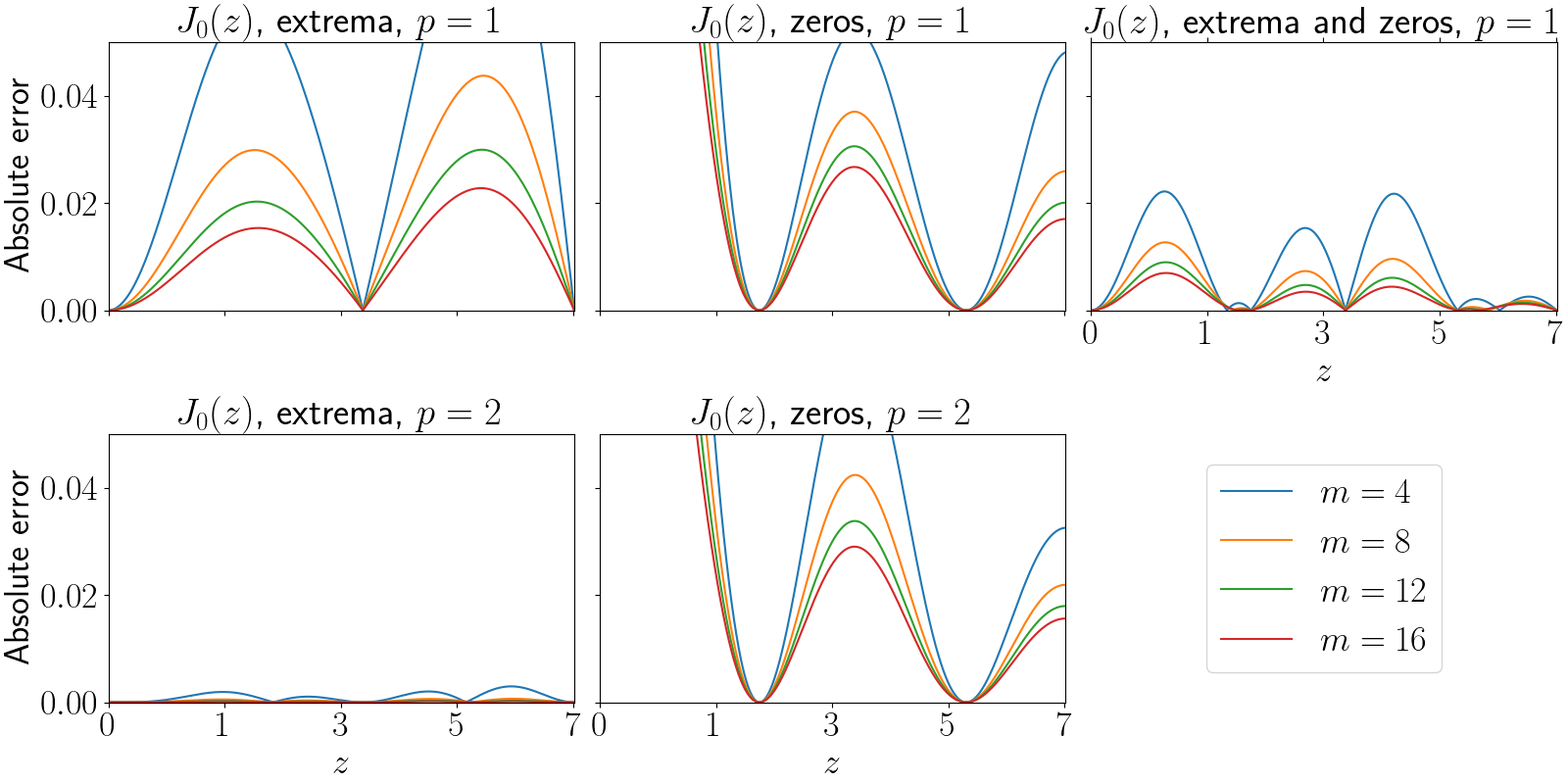}
		\end{subfigure}%
		\vspace{0.025\linewidth}
		\begin{subfigure}{\textwidth}
			\centering
			\includegraphics[width=\linewidth]{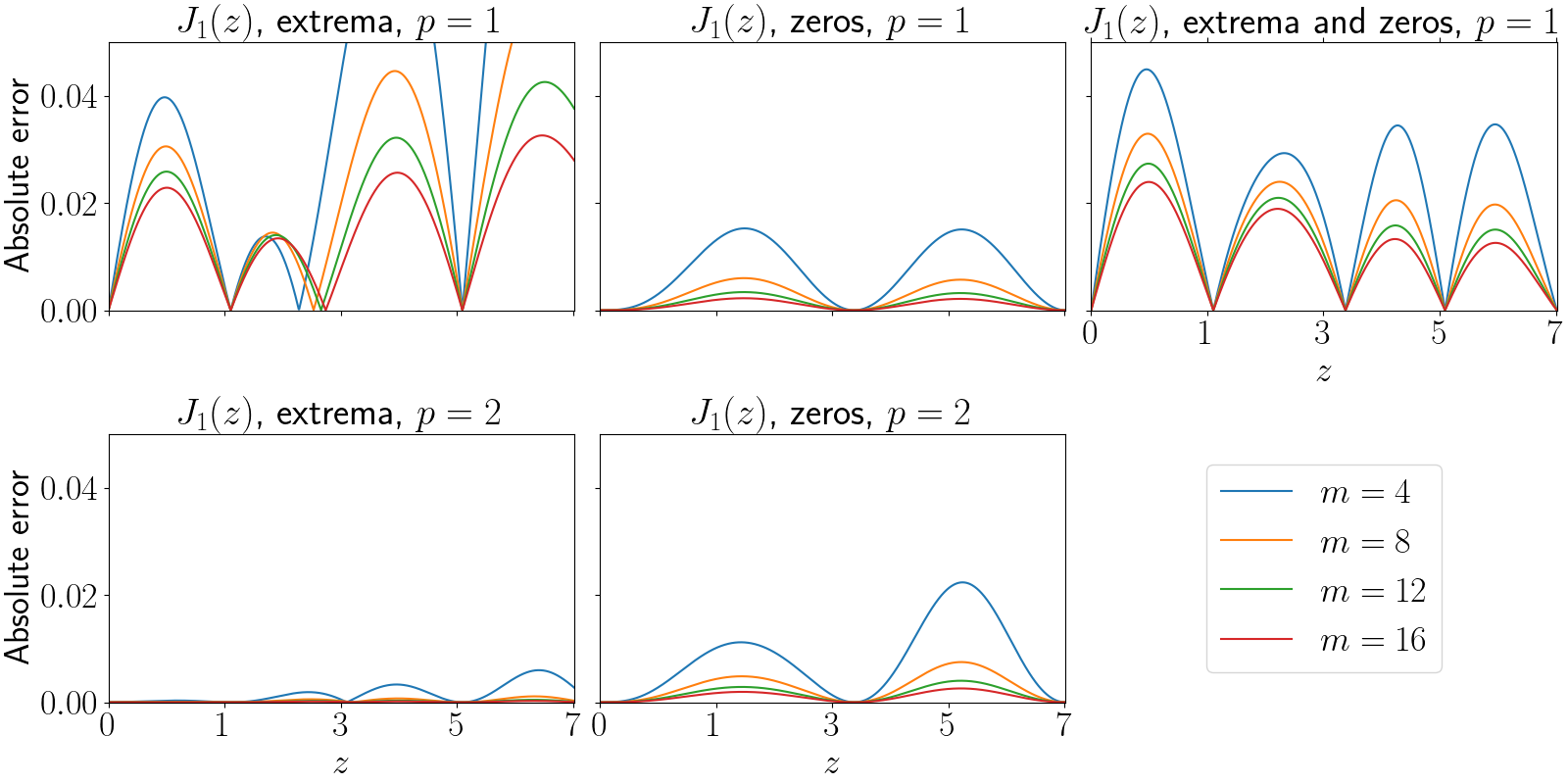}
		\end{subfigure}
		\caption{Here we compare the absolute error of approximate interpolations of $J_0(z)$ and $J_1(z)$ for different choices of nodes. These are indexed by $m$ such that the interpolation includes all nodes between $-z_m$ and $z_m$.}
		\label{fig:bessel-plot-comparison}
	\end{figure}
	
	\begin{figure}[h!]
		\centering
		\includegraphics[width=\linewidth]{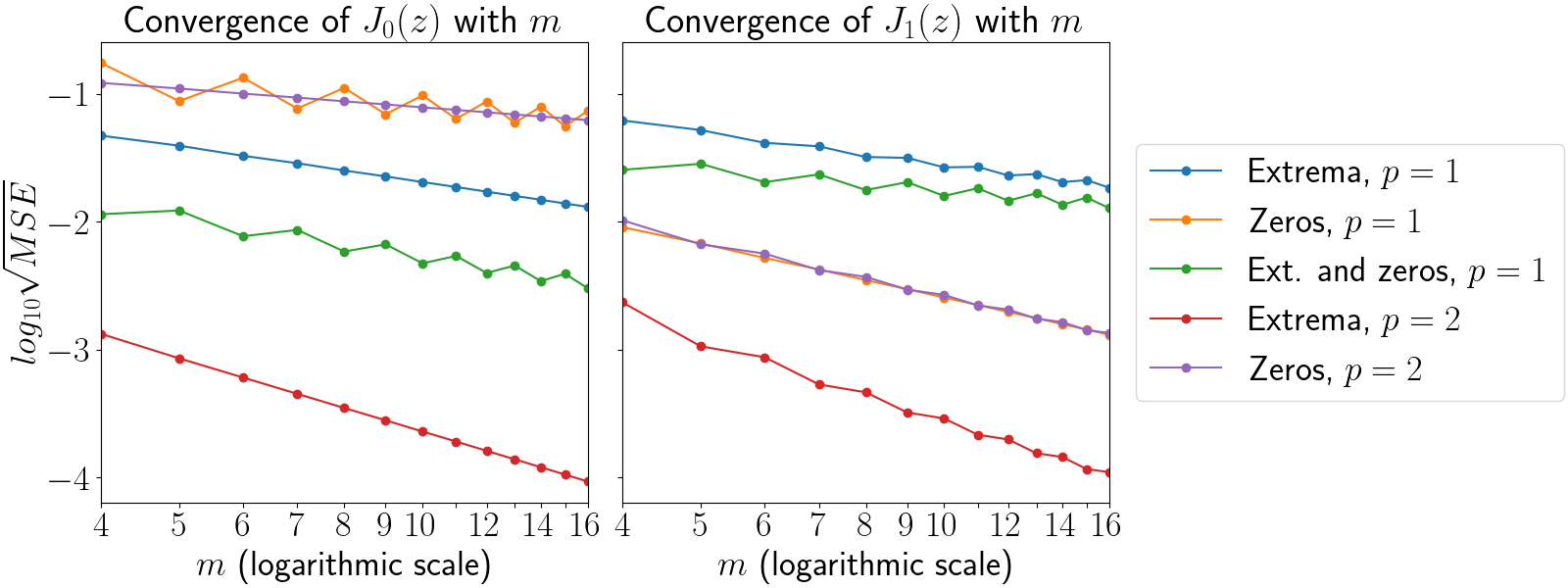}
		\caption{Here we compare the convergence of approximations of $J_0(z)$ and $J_1(z)$ with the number of nodes included in the approximation, where each includes all nodes from $-z_m$ to $z_m$. The mean square error is calculated over the same interval shown in Fig.~\ref{fig:bessel-plot-comparison}.}
		\label{fig:bessel-plot-convergence}
	\end{figure}
	
	We also give fits for the convergence exponent (i.e., $\alpha$ in $\sqrt{MSE} \propto R_m^{-\alpha}$ as plotted in Fig.~\ref{fig:bessel-plot-convergence}) in Table~\ref{tbl:bessel-scaling-exponents}. For comparison, we also give the estimates from the analytic remainder integrals, which by Prop.~\ref{prop:F-error-convergence-rate} set the convergence rate. The precision of these estimates depends on the sharpness of our bounds on the analytic remainder integrals.
	The simple estimates we have used are generally close to the fitted result, but since the bounds used in our previous calculations were relatively loose, the actual convergence rate significantly exceeds the prediction in a few cases.
	
	\begin{table}[b!]
		\centering
		\begin{tabular}{c || c | c || c | c}
			\multicolumn{5}{c}{Error scaling exponents on $R_m$} \\
			\hline\hline
			& \multicolumn{2}{c||}{Fit} & \multicolumn{2}{c}{Estimate} \\
			\hline
			& Even & Odd & Even & Odd \\
			\hline\hline
			Extrema, $p=1$ & 0.93 & 0.85 & 1 & 0.5 \\
			\hline
			Zeros, $p=1$ & 0.56 & 1.40 & 0.5 & 1 \\
			\hline
			Extrema and zeros, $p=1$ & 0.98 & 0.51 & 1 & 0.5 \\
			\hline
			Extrema, $p=2$ & 1.91 & 2.13 & 1.5 & 2 \\
			\hline
			Zeros, $p=2$ & 0.49 & 1.44 & 0.5 & 1.5
		\end{tabular}
		\caption{Here we compare the slopes of linear fits of Fig.~\ref{fig:bessel-plot-convergence} to the exponents $\alpha$ on $R_m^{-\alpha}$ from the analytic remainder integrals.
		We have also sharpened some bounds from the simpler forms presented earlier; for example, the bound in Eq.~(\ref{eq:bessel-remainder-2-num-right}) can be sharpened using antisymmetry in $y$.}
		\label{tbl:bessel-scaling-exponents}
	\end{table}
	
	Finally, in Figure~\ref{fig:bessel-expansion-comparison} we compare the approximate barycentric expansion weighted at extrema with $p=2$ against the Taylor expansion with the same number of terms and the asymptotic expansion.
	The barycentric expansion has the advantage of neither diverging as $|z|$ grows large nor as $|z|$ becomes small; in fact, outside the region delineated by the nodes, the barycentric expression tends to the correct large-$|z|$ limit of $0$ due to the analytic remainder in the denominator.
	Whereas the Taylor expansion prioritizes accuracy for small $|z|$ and the asymptotic expansion prioritizes accuracy for large $|z|$, the approximate barycentric expansion has relatively consistent accuracy across the interpolation region, trading specialization for increased generality.
	This is especially consequential for higher orders of Bessel functions, for which the asymptotic expansion is inaccurate for larger regions around $z=0$. 
	
	\begin{figure}[t!]
		\centering
		\includegraphics[width=\linewidth]{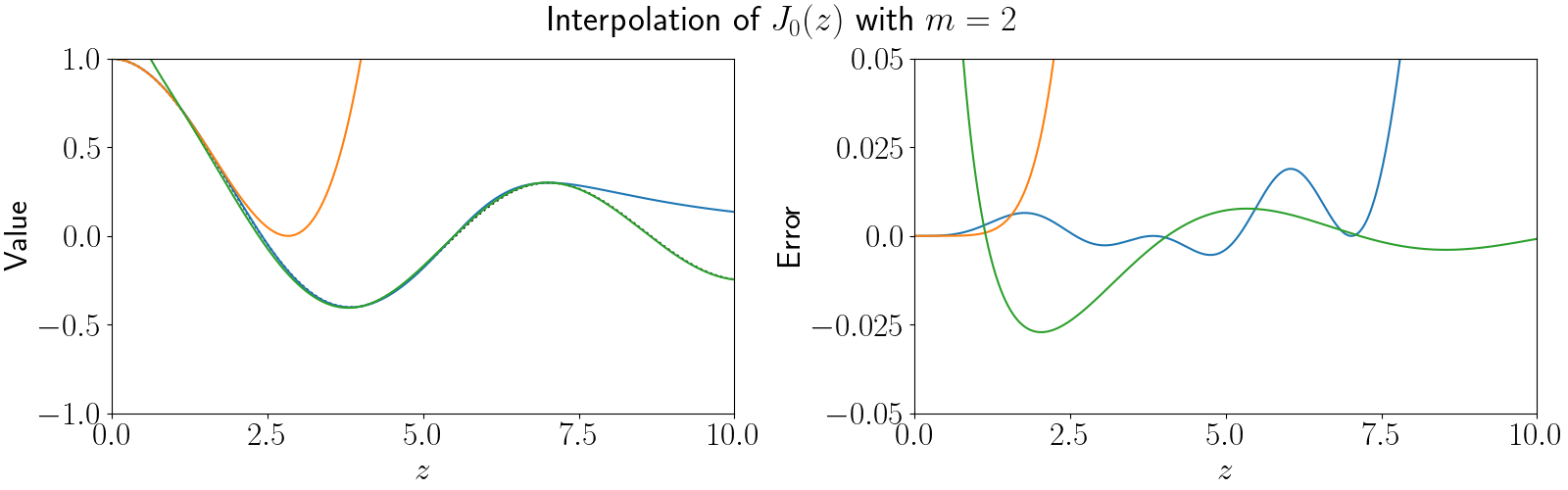}
		\includegraphics[width=\linewidth]{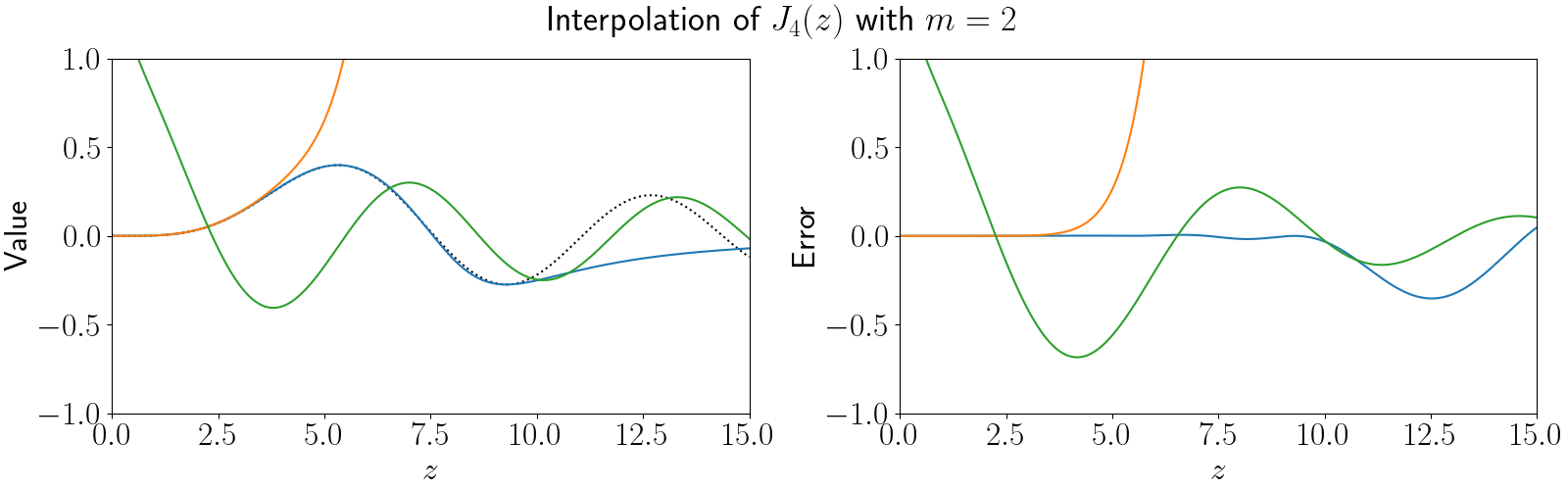}
		\includegraphics[width=\linewidth]{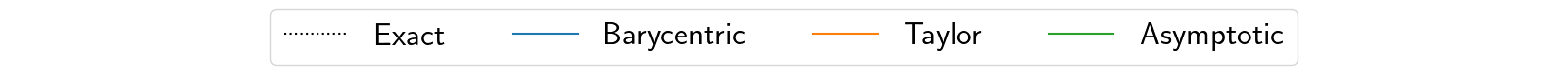}
		\caption{Here we compare our barycentric approximations for $J_0(z)$ and $J_4(z)$ with the corresponding Taylor series and asympltotic approximations. The left plot compares the approximations directly and the right plot compares their error. For the barycentric expansion we include as nodes all extrema from $-z_m$ to $z_m$, and to match the number of independent terms, for the Taylor series we truncate after the third term.}
		\label{fig:bessel-expansion-comparison}
	\end{figure}
	
	\section{Conclusion}
	
	We have demonstrated a method to produce function expansions in the form of barycentric rational interpolations using standard complex analysis techniques. These expansions allow functions to be approximated with good accuracy in any region from only a few explicit evaluations at the surrounding nodes.
	We have demonstrated the method with cosine and the Bessel functions of the first kind and shown that the error and convergence of these approximations is readily quantifiable.
	In principle, any nodes many be chosen to construct these barycentric expansions. Our main interest in this paper has been the analytic properties of the expansions and so we have limited our focus to natural points for each function such as extrema and zeros.
	
	\section*{Acknowledgments}
	
	I would like to thank Matthew S. Mizuhara for helpful comments on drafts of this paper.
	
	\printbibliography
	
	\newpage
	
	\appendix
	
	\section{Barycentric rational expansions for cosine} \label{sec:cosine-different-anchors}
	
	Here we give the explicit expressions for the cosine expansions compared in Sec.~\ref{sec:cos-denom-not-extrema}, except those which were presented earlier in Secs.~\ref{sec:first-cos-expansion} and \ref{sec:cos-denom-power-weights}.
	
	The expansions of $\cos z$ weighted at zeros are
	\begin{align}
		\cos z
		& = \frac{\sfrac{\cos z}{\cos z}}{\sfrac{1}{\cos z}}
		= \frac{1}{\sum_{k=\infty}^\infty \frac{(-1)^k}{z-\pi\left(k-\frac{1}{2}\right)}} \, , \label{eq:cos-zeros-anchors-p=1} \\
		\cos z
		& = \frac{\sfrac{\cos z}{\cos^2 z}}{\sfrac{1}{\cos^2 z}}
		= \frac{\sum_{k=\infty}^\infty \frac{(-1)^k}{z-\pi\left(k-\frac{1}{2}\right)}}{\sum_{k=\infty}^\infty \frac{1}{\left(z-\pi\left(k-\frac{1}{2}\right)\right)^2}} \, . \label{eq:cos-zeros-anchors-p=2}
	\end{align}
	The expansions weighted at both extrema and zeros are
	\begin{align}
		\cos z
		& = \frac{\sfrac{\cos z}{\sin 2z}}{\sfrac{1}{\sin 2z}}
		= \frac{ \sum_{k=-\infty}^\infty \frac{(-1)^k}{z - \pi k} }
		{ \sum_{k=-\infty}^\infty \frac{(-1)^k}{z - \frac{\pi}{2} k} } \, , \label{eq:cos-extrema-and-zeros-anchors-p=1} \\
		\cos z
		& = \frac{\sfrac{\cos z}{\sin^2 2z}}{\sfrac{1}{\sin^2 2z}}
		= \frac{ \sum_{k=-\infty}^\infty \left( \frac{(-1)^k}{\left(z - \pi k\right)^2} - \frac{(-1)^k}{z - \pi \left(k+\frac{1}{2}\right)} \right) }{ \sum_{k=-\infty}^\infty \frac{1}{\left(z - \frac{\pi}{2} k\right)^2} } \, .
	\end{align}
	The expansions weighted at midpoints between extrema and zeros are
	\begin{align}
		\cos z
		& = \frac{\sfrac{\cos z}{\sin \left(2\left(z-\frac{\pi}{4}\right)\right)}}{\sfrac{1}{\sin \left(2\left(z-\frac{\pi}{4}\right)\right)}}
		= \frac{\sum_{k=-\infty}^\infty
			\frac{1}{\sqrt{2}}\frac{(-1)^{\left\lfloor\sfrac{k}{2}\right\rfloor}}{z - \frac{\pi}{2}\left(k+\frac{1}{2}\right)}
		}
		{\sum_{k=-\infty}^\infty
			\frac{1}{z - \frac{\pi}{2}\left(k+\frac{1}{2}\right)}
		} \, , \\
		\cos z
		& = \frac{\sfrac{\cos z}{\sin^2 \left(2\left(z-\frac{\pi}{4}\right)\right)}}{\sfrac{1}{\sin^2 \left(2\left(z-\frac{\pi}{4}\right)\right)}} \\
		& = \frac{ \sum_{k=-\infty}^\infty \left(
			\frac{1}{\sqrt{2}}\frac{ (-1)^{\left\lceil\sfrac{k}{2}\right\rceil} }{\left(z - \frac{\pi}{2}\left(k+\frac{1}{2}\right)\right)^2}
			-
			\frac{1}{\sqrt{2}}\frac{ (-1)^{\left\lfloor\sfrac{k}{2}\right\rfloor} }{z - \frac{\pi}{2}\left(k+\frac{1}{2}\right)}
			\right)}{\sum_{k=-\infty}^\infty
			\frac{1}{\left(z - \frac{\pi}{2}\left(k+\frac{1}{2}\right)\right)^2}
		} \, . \nonumber
	\end{align}
	
	\section{Barycentric rational expansions for the Bessel functions} \label{sec:bessel-expansions-calculation}
	
	Here we give the full details necessary to calculate the Bessel function expansions in Sec.~\ref{sec:bessel-barycentric}. We reproduce some demonstrative calculations from that section here for completeness.
	
	\subsection{Bessel functions weighted at extrema with $p=1$} \label{sec:bessel-p-1}
	
	First we consider weighting at the extrema with $p=1$, writing
	\begin{equation} \label{eq:bessel-weighting-1}
		J_q (z) = \frac{ \sfrac{J_q (z)}{\frac{d}{dz}J_q (z)} }{ \sfrac{1}{\frac{d}{dz}J_q (z)} } \, .
	\end{equation}
	As we will show in the remainder of this section, the result is
	\begin{equation} \label{eq:bessel-weighting-1-final-result}
		J_q(z) = \frac{A^{(0,num)}(z) + \sum_{k=-\infty,k\ne0}^\infty \frac{a_{-1}^{(k,num)}}{z-z_k}}{ A^{(0,den)}(z) + \sum_{k=-\infty,k\ne0}^\infty \frac{a_{-1}^{(k,den)}}{z-z_k} }
	\end{equation} 
	where $z_k$ are the zeros of the Bessel function aside from $z_0=0$, the Laurent principal part coefficients for $|k|\ge1$ are given by Eq.~(\ref{eq:bessel-laurent-p1}), and the Laurent principal part for the pole at $z_0=0$ is
	\begin{equation}
		\begin{matrix}
			A^{(0,num)}(z) = \left\{\begin{matrix}
				\vphantom{-\frac{2}{z}} \\
				\vphantom{0}
			\end{matrix}\right. \vspace{1.5mm} \\
			A^{(0,den)}(z) = \left\{\begin{matrix}
				\vphantom{-\frac{2}{z}} \\
				\vphantom{0} \\
				\vphantom{\sum_{n=1}^{q-1} \frac{a_{-n}^{(0,den)}}{z^n}}
			\end{matrix}\right.
		\end{matrix}
		\begin{matrix*}[l]
			-\frac{2}{z} \, , & q = 0 \\
			0 \, , & q \ge 1 \vspace{1.5mm} \\
			-\frac{2}{z} \, , & q = 0 \\
			0 \, , & q = 1 \\
			\sum_{n=1}^{q-1} \frac{a_{-n}^{(0,den)}}{z^n} \, , & q \ge 2
		\end{matrix*}
	\end{equation}
	where the coefficients $a_{-n}^{(0,den)}$ are calculated by solving the matrix equation Eq.~(\ref{eq:laurent-coeff-mat-eq-simple-appendix}).
	
	\subsubsection{Calculating the Laurent series principal parts} \label{sec:bessel-laurent-p1}
	
	The numerator of Eq.~(\ref{eq:bessel-weighting-1}) has no pole at $z_0=0$ and simple poles at all other $z_k$, while the denominator has a pole of order $q-1$ at $z_0$ for $q\ge2$ and simple poles at all other $z_k$.
	As described in Sec.~\ref{sec:bessel-laurent}, to calculate the Laurent principal part for the denominator's high-order pole at $z_0$, we need to solve a matrix equation with the form
	\begin{equation} \label{eq:laurent-coeff-mat-eq-simple-appendix}
		c = B a
	\end{equation}
	where the vector and matrix elements $[a]_j$, $[B]_{j,j'}$, $[c]_j$ are given by
	\begin{align}
		\begin{matrix*}
			\left[a\right]_j & = & a_{-(j_* - j_0) + j} \, , \\
			\left[B\right]_{j,j'} & = & b_{j_* + j - j'} \Theta\left(j-j'\right) \, , \\
			\left[c\right]_j & = & c_{j_0 + j} \, ,
		\end{matrix*}
		\left.\vphantom{\begin{matrix*}
				\left[a\right]_j & = & a_{-(j_* - j_0) + j} \, , \\
				\left[B\right]_{j,j'} & = & b_{j_* + j - j'} \Theta\left(j-j'\right) \, , \\
				\left[c\right]_j & = & c_{j_0 + j} \, ,
			\end{matrix*}}\right\}
		\quad 0 \le j,j' < j_* - j_0 - 1
	\end{align}
	To apply Eq.~(\ref{eq:laurent-coeff-mat-eq-simple-appendix}) to the denominator for $q\ge2$, we take $j_0=0$, $j_*=q-1$, $c_j=\delta_{j,0}$, and from the Bessel function power series, $b_{q-1+2j} = \frac{(-1)^j}{2^{q+2j}} \frac{q+2j}{j!(q+j)!}$ for $j\ge0$ with all other $b$ terms vanishing.
	
	The poles at $z_k$ for $k\ne0$ (i.e., $z_k\ne0$) are all simple, in which case the method of Eq.~(\ref{eq:laurent-coeff-mat-eq-simple-appendix}) reduces to a single equation that is trivial to solve analytically. The resulting Laurent coefficients for the numerator $\sfrac{J_q (z)}{\frac{d}{dz}J_q (z)}$ and denominator $\sfrac{1}{\frac{d}{dz}J_q (z)}$ for $k\ne0$ are
	\begin{equation}\begin{aligned} \label{eq:bessel-laurent-p1}
			a_{-1}^{(k,num)}
			& = \left.\frac{J_q(z)}{\frac{d^2}{dz^2}J_q(z)}\right|_{z=z_k}
			= \frac{4J_q(z_k)}{J_{q-2}(z_k)-2J_q(z_k)+J_{q+2}(z_k)} \, , \\
			a_{-1}^{(k,den)}
			& = \left.\frac{1}{\frac{d^2}{dz^2}J_q(z)}\right|_{z=z_k}
			= \frac{4}{J_{q-2}(z_k)-2J_q(z_k)+J_{q+2}(z_k)} \, .
	\end{aligned}\end{equation}
	
	\subsubsection{Calculating the analytic remainders} \label{sec:bessel-analytic-remainder-p1}
	
	As noted in Sec.~\ref{sec:bessel-analytic-remainders}, by the parity of the Bessel functions, we only need to calculate for only the integrals along the right and top sides, $S=R,T$,
	\begin{equation} \label{eq:bessel-p-1-integral-prototypes}
		I_{(S,m)}^{(num,j)} = \int_{\Gamma_m^{(S)}} dw \frac{1}{w^{j+1}} \frac{J_q(w)}{\frac{d}{dw}J_q(w)} \, , \quad
		I_{(S,m)}^{(den,j)} = \int_{\Gamma_m^{(S)}} dw \frac{1}{w^{j+1}} \frac{1}{\frac{d}{dw}J_q(w)} \, .
	\end{equation}
	
	We will begin with the right side, $\Gamma_m^{(R)}$. As derived in Eqs.~(\ref{eq:bessel-geometric-series-right-side-example}) and (\ref{eq:bessel-geo-series-right-num-example}),
	\begin{align}
		\left.\frac{1}{\frac{d}{dw}J_q(w)}\right|_{R_m+iy}
		& = -\sqrt{\frac{\pi \left(R_m+iy\right)}{2}} \frac{(-1)^m}{\cosh y}\left(1 + O\left(R_m^{-1}\right)\right) \, , \label{eq:bessel-geometric-series-right-side} \\
		\left.\frac{J_q(w)}{\frac{d}{dw}J_q(w)}\right|_{R_m+iy}
		& = i \tanh y + O\left(R_m^{-1}\right) \, . \label{eq:bessel-geo-series-right-num}
	\end{align}

	For $I_{(R,m)}^{(num,j)}$, use the ML inequality together with $\left|\tanh y\right| \le 1$ to see that
	\begin{equation} \label{eq:bessel-remainder-1-num-right}
		\left|I_{(R,m)}^{(num,j)}\right|
		= \left|\int_{\Gamma_m^{(R)}} dw \frac{1}{w^{j+1}} \frac{J_q(w)}{\frac{d}{dw}J_q(w)}\right|
		\le \frac{2}{R_m^j} + O\left(R_m^{-j-1}\right) \, , 
	\end{equation}
	The entire term vanishes for $j\ge1$ and the error term vanishes for all $j$. The $j=0$ case must be treated individually, but since here the integrand $\sfrac{J_q(w)}{\frac{d}{dw}J_q(w)}$ is odd while $j$ is even, by the parity argument of Prop.~\ref{prop:analytic-remainder-parity-argument}, $I_{(R,m)}^{(num,j)}$ cancels with $I_{(L,m)}^{(num,j)}$ so that this integral does not contribute to the analytic remainder.
	
	Next, using the triangle inequality, $I_{(R,m)}^{(den,j)}$ tends to $0$ as $m\to\infty$ for all $j$.
	\begin{equation}\begin{aligned} \label{eq:bessel-remainder-1-den-right}
		\left|I_{(R,m)}^{(den,j)}\right|
		& = \left|\int_{\Gamma_m^{(R)}} dw \frac{1}{w^{j+1}} \frac{1}{\frac{d}{dw}J_q(w)}\right| \\
		& < \sqrt{\frac{\pi}{2}} \frac{1}{R_m^{j+\sfrac{1}{2}}} \int_{-R_m}^{R_m} dy \frac{1}{\cosh y} + O\left(R_m^{-j-\sfrac{1}{2}}\right) \\
		& < \sqrt{\frac{\pi}{2}} \frac{1}{R_m^{j+\sfrac{1}{2}}} \int_{-R_m}^{R_m} dy \, e^{-|y|} + O\left(R_m^{-j-\sfrac{1}{2}}\right) \\
		& = \frac{\sqrt{2\pi}}{R_m^{j+\sfrac{1}{2}}} \left(1-e^{-R_m}\right) + O\left(R_m^{-j-\sfrac{1}{2}}\right) \, .
	\end{aligned}\end{equation}
	This integral vanishes in the limit $m\to\infty$ for any $j$.
	
	For the top side of the contour, we will instead observe that
	\begin{equation}\begin{aligned} \label{eq:bessel-top-side-simplification}
		\sin\left(x+iR_m-\frac{\pi}{2}\left(q+\frac{1}{2}\right)\right)
		& = e^{R_m} O\left(1\right) \, , \\
		\cos\left(x+iR_m-\frac{\pi}{2}\left(q+\frac{1}{2}\right)\right)
		& = e^{R_m}  O\left(1\right) \, .
	\end{aligned}\end{equation}
	Again using these in the asymptotic forms of Eq.~(\ref{eq:bessel-asymptotic}), we find
	\begin{align}
		\left.\frac{1}{\frac{d}{dw}J_q(w)}\right|_{x+iR_m} \!\!\!\!
		& = e^{-R_m} O\left(R_m^{\sfrac{1}{2}}\right) \, , \label{eq:bessel-geo-series-top-den} \\
		\left.\frac{J_q(w)}{\frac{d}{dw}J_q(w)}\right|_{x+iR_m} \!\!\!\!
		& = O\left(1\right) \, . \label{eq:bessel-geo-series-top-num}
	\end{align}
	
	Both of these are independent of $x$, so using $\left|R_m+iy\right|^{-j} \le R_m^{-j}$ with the triangle inequality, the integrals $I_{(T,m)}^{(num,j)}$ and $I_{(T,m)}^{(num,j)}$ are simple to bound.
	\begin{align}
		\left|I_{(T,m)}^{(num,j)}\right|
		& = \left|\int_{\Gamma_m^{(T)}} dw \frac{1}{w^{j+1}} \frac{J_q(w)}{\frac{d}{dw}J_q(w)}\right|
		= O\left(R_m^{-j}\right) \, , \label{eq:bessel-remainder-1-num-top} \\
		\left|I_{(T,m)}^{(den,j)}\right|
		& = \left|\int_{\Gamma_m^{(T)}} dz \frac{1}{w^{j+1}} \frac{1}{\frac{d}{dw}J_q(zw)}\right|
		= e^{-R_m} O\left(R_m^{-j+\sfrac{1}{2}}\right) \, . \label{eq:bessel-remainder-1-den-top}
	\end{align}
	$I_{(T,m)}^{(num,j)}$ vanishes as $m\to\infty$ for any $j$, and $I_{(T,m)}^{(num,j)}$ vanishes as $m\to\infty$ for $j\ge1$. For the case of $j=0$, we again note that its integrand $\sfrac{J_q(w)}{\frac{d}{dw}J_q(w)}$ is odd while $j$ is even, so that by the parity argument of Prop.~\ref{prop:analytic-remainder-parity-argument}, $I_{(T,m)}^{(num,j)}$ cancels with $I_{(B,m)}^{(num,j)}$ and this integral never contributes to the analytic remainder.
	
	Hence all analytic remainder integrals either vanish in the limit $m\to\infty$ or cancel with one another. Thus, as the integration contour $\Omega_m$ grows to cover the entire complex plane and include all the poles of the numerator and denominator, the analytic remainders of both the numerator and denominator vanish, leaving only contributions from the Laurent series principal parts of each of their poles in Eq.~(\ref{eq:bessel-weighting-1-final-result}).
	
	\subsection{Bessel functions weighted at extrema with $p=2$} \label{sec:bessel-p-2}
	
	Now we consider the case $p=2$, writing
	\begin{equation} \label{eq:bessel-weighting-2}
		J_q (z) = \frac{ \sfrac{J_q (z)}{\left(\frac{d}{dz}J_q (z)\right)^2} }{ \sfrac{1}{\left(\frac{d}{dz}J_q (z)\right)^2} } \, .
	\end{equation}
	The resulting expansion is
	\begin{equation} \label{eq:bessel-weighting-2-final-result}
		J_q(z) = \frac{A^{(0,num)}(z) + \sum_{k=-\infty, k\ne0}^\infty \left(\frac{a_{-2}^{(k,num)}}{\left(z-z_k\right)^2} + \frac{a_{-1}^{(k,num)}}{z-z_k}\right)}
		{1 + A^{(0,den)}(z) + \sum_{k=-\infty,k\ne0}^\infty \left(\frac{a_{-2}^{(k,den)}}{\left(z-z_k\right)^2} + \frac{a_{-1}^{(k,den)}}{z-z_k}\right) }
	\end{equation}
	where $z_k$ are the zeros of the Bessel function aside from $z_0=0$, the Laurent principal part coefficients for $|k|\ge1$ are given by Eq.~(\ref{eq:bessel-laurent-p2}), and the Laurent principal part for $k=0$ is
	\begin{equation}
		\begin{matrix}
			A^{(0,num)}(z) = \left\{\begin{matrix}
				\vphantom{\frac{4}{z^2}} \\
				\vphantom{0} \\
				\vphantom{\sum_{n=1}^{q-2} \frac{a_{-n}^{(0,num)}}{z^n}}
			\end{matrix}\right. \vspace{1.5mm} \\
			A^{(0,den)}(z) = \left\{\begin{matrix}
				\vphantom{\frac{4}{z^2}} \\
				\vphantom{0} \\
				\vphantom{\sum_{n=1}^{2(q-1)} \frac{a_{-n}^{(0,den)}}{z^n}}
			\end{matrix}\right.
		\end{matrix}
		\begin{matrix*}[l]
			\frac{4}{z^2} \, , & q = 0 \\
			0 \, , & q = 1 \\
			\sum_{n=1}^{q-2} \frac{a_{-n}^{(0,num)}}{z^n} \, , & q \ge 2 \vspace{1.5mm} \\
			\frac{4}{z^2} \, , & q = 0 \\
			0 \, , & q = 1 \\
			\sum_{n=1}^{2(q-1)} \frac{a_{-n}^{(0,den)}}{z^n} \, , & q \ge 2
		\end{matrix*}
	\end{equation}
	where the coefficients $a_{-n}^{(0,num)}$ and $a_{-n}^{(0,den)}$ are calculated by solving the matrix equation Eq.~(\ref{eq:laurent-coeff-matrix-equation-p2}).
	
	\subsubsection{Calculating the Laurent series principal parts} \label{sec:bessel-laurent-p2}
	
	Using the approach of Eq.~(\ref{eq:laurent-coeff-matrix-equation-second-order-denom}) in Sec.~\ref{sec:bessel-laurent} with $B_0 = B_1 = B$, we must solve
	\begin{equation} \label{eq:laurent-coeff-matrix-equation-p2}
		c = B^2 a
	\end{equation}
	with
	\begin{align}
		\begin{matrix*}
			\left[a\right]_j & = & a_{-(2j_* - j_0) + j} \, , \\
			\left[B\right]_{j,j'} & = & b_{j_* + j - j'} \Theta\left(j-j'\right) \, , \\
			\left[c\right]_j & = & c_{j_0 + j} \, .
		\end{matrix*}
		\left.\vphantom{\begin{matrix*}
				\left[a\right]_j & = & a_{-(2j_* - j_0) + j} \, , \\
				\left[B\right]_{j,j'} & = & b_{j_* + j - j'} \Theta\left(j-j'\right) \, , \\
				\left[c\right]_j & = & c_{j_0 + j} \, .
		\end{matrix*}}\right\}
		\quad 0 \le j,j' < 2 j_* - j_0 - 1
	\end{align}
	for both the numerator and denominator of Eq.~(\ref{eq:bessel-weighting-2}).
	For the numerator, we have $j_0=q$ and $c_{q+2j} = \frac{(-1)^j}{2^{q+2j}} \frac{1}{j!(q+j)!}$ for $j\ge0$.
	For the denominator, we have $j_0=0$ and $c_j = \delta_{j,0}$.
	In both cases, $j_*=q-1$ and $b_{q-1+2j} = \frac{(-1)^j}{2^{q+2j}} \frac{q+2j}{j!(q+j)!}$ for $j\ge0$.
	All coefficients not covered by the indices given here vanish.
	
	Applying the same method to find the Laurent principal part coefficients for the second-order poles at $z_k$ for $k\ne0$ and using the notation $\frac{d^n}{dz^n}f(z) = f^{(n)}(z)$ for brevity, we find
	\begin{equation}\begin{aligned} \label{eq:bessel-laurent-p2}
			a_{-2}^{(k,num)}
			& = \frac{J_q(z_k)}{(J_q^{(2)}(z_k))^2} \, , \quad
			a_{-1}^{(k,num)}
			= \frac{J_q^{(1)}(z_k)}{(J_q^{(2)}(z_k))^2} - \frac{(J_q(z_k))(J_q^{(3)}(z_k))}{(J_q^{(2)}(z_k))^3} \, , \\
			a_{-2}^{(k,den)}
			& = \frac{1}{(J_q^{(2)}(z_k))^2} \, , \quad
			a_{-1}^{(k,den)}
			= - \frac{J_q^{(3)}(z_k)}{(J_q^{(2)}(z_k))^3} \, .
	\end{aligned}\end{equation}
	If desired, these expressions can be expanded using the recurrence relation Eq.~(\ref{eq:bessel-recurrence}) as in Eq.~(\ref{eq:bessel-laurent-p1}), although the result is much more cumbersome.
	
	\subsubsection{Calculating the analytic remainders} \label{sec:bessel-analytic-remainder-p2}
	
	As in Sec.~\ref{sec:bessel-analytic-remainder-p1}, we will calculate the integrals for the right and top sides of the contour $\partial\Omega_m$, $S=R,T$, and the full contour follows due to the parity of the Bessel functions.
	\begin{equation}\begin{aligned}
		I_{(S,m)}^{(num,j)} & = \int_{\Gamma_m^{(S)}} dw \frac{1}{w^{j+1}} \frac{J_q(w)}{\left(\frac{d}{dw}J_q(w)\right)^2} \, , \\
		I_{(S,m)}^{(den,j)} & = \int_{\Gamma_m^{(S)}} dw \frac{1}{w^{j+1}} \frac{1}{\left(\frac{d}{dw}J_q(w)\right)^2} \, .
	\end{aligned}\end{equation}
	
	Squaring Eq.~(\ref{eq:bessel-geometric-series-right-side}) and multiplying Eqs.~(\ref{eq:bessel-geometric-series-right-side}) and (\ref{eq:bessel-geo-series-right-num}), we have
	\begin{align}
			\frac{1}{\left(\frac{d}{dw}J_q(w)\right)^2}
			& = \frac{\pi \left(R_m+iy\right)}{2} \frac{1}{\cosh^2 y}\left(1 + O\left(R_m^{-1}\right)\right) \, , \label{eq:bessel-geometric-series-right-side-squared} \\
			\frac{J_q(w)}{\left(\frac{d}{dw}J_q(w)\right)^2}
			& = i (-1)^{m+1} \sqrt{\frac{\pi\left(R_m+iy\right)}{2}} \frac{\sinh y}{\cosh^2 y} + O\left(R_m^{-\sfrac{1}{2}}\right)
	\end{align}
	
	By the triangle inequality,
	\begin{equation}\begin{aligned} \label{eq:bessel-remainder-2-num-right}
			\left|I_{(R,m)}^{(num,j)}\right|
			& = \left|\int_{\Gamma_m^{(R)}} dw \frac{1}{w^{j+1}} \frac{J_q(w)}{\left(\frac{d}{dw}J_q(w)\right)^2}\right| \\
			& < \int_{-R_m}^{R_m} dy \left( \sqrt{\frac{\pi}{2}} \frac{1}{R_m^{j+\sfrac{1}{2}}} \frac{\left|\sinh y\right|}{\cosh^2 y} + O\left(R_m^{-j-\sfrac{3}{2}}\right) \right) \\
			& = \frac{\sqrt{2\pi}}{R_m^{j+\sfrac{1}{2}}} \left(1 - \text{sech} R_m\right) + O\left(R_m^{-j-\sfrac{1}{2}}\right) \, ,
	\end{aligned}\end{equation}
	which vanishes as $m\to\infty$ for any $j$.
	
	As in Sec.~\ref{sec:bessel-analytic-remainders}, the triangle inequality shows that the denominator integrals vanish in the limit $m\to\infty$ for $j\ge1$,
	\begin{align} \label{eq:bessel-remainder-2-den-right}
		\left|I_{(R,m)}^{(den,j)}\right|
		& = \left|\int_{\Gamma_m^{(R)}} dw \frac{1}{w^{j+1}} \frac{1}{\left(\frac{d}{dw}J_q(w)\right)^2}\right| \\
		& \le \frac{\pi}{2 R_m^j} \int_{-R_m}^{R_m} dy \frac{1}{\cosh^2 y} \left(1+O\left(R_m^{-1}\right)\right)
		= \pi R_m^{-j}\left(1+O\left(R_m^{-1}\right)\right) \, . \nonumber
	\end{align}
	However, in the case $j=0$ the integral does not vanish, instead giving
	\begin{equation}\begin{aligned} \label{eq:bessel-remainder-2-den-right-j0}
		I_{(R,m)}^{(den,0)}
		& = \int_{\Gamma_m^{(R)}} dz \frac{1}{w} \frac{1}{\left(\frac{d}{dw}J_q(w)\right)^2} \\
		& = i \frac{\pi}{2} \int_{-R_m}^{R_m} dy \frac{1}{\cosh^2 y} \left(1+O\left(R_m^{-1}\right)\right) = i\pi + O\left(R_m^{-1}\right) \, .
	\end{aligned}\end{equation}
	
	For the top side fo the contour, we square Eq.~(\ref{eq:bessel-geo-series-top-den}) to obtain
	\begin{equation}
		\frac{1}{\left(\frac{d}{dw}J_q(w)\right)^2}
		 = e^{-2R_m} O\left(R_m\right) \, , \quad\quad \label{eq:bessel-geometric-series-top-side-squared}
		\frac{J_q(w)}{\left(\frac{d}{dw}J_q(w)\right)^2}
		 = e^{-R_m} O\left(R_m^{\sfrac{1}{2}}\right) \, .
	\end{equation}
	In the same way as in Eqs.~(\ref{eq:bessel-remainder-1-num-top}) and (\ref{eq:bessel-remainder-1-den-top}), the resulting integrals vanish in the limit $m\to\infty$ for any $j$ due to the exponential factors,
	\begin{align}
		\left|I_{(T,m)}^{(num,j)}\right|
		& = \left|\int_{\Gamma_m^{(T)}} dw \frac{1}{w^{j+1}} \frac{J_q(w)}{\left(\frac{d}{dw}J_q(w)\right)^2}\right|
		= e^{-R_m} O\left(R_m^{-j+\sfrac{1}{2}}\right) \, , \\
		\left|I_{(T,m)}^{(den,j)}\right|
		& = \left|\int_{\Gamma_m^{(T)}} dw \frac{1}{w^{j+1}} \frac{1}{\left(\frac{d}{dw}J_q(w)\right)^2}\right|
		= e^{-2R_m} O\left(R_m^{-j+1}\right) \, .
	\end{align}
	
	In the limit of larger $m$, the only nonvanishing analytic remainder integrals are the $j=0$ top and bottom side integrals for the denominator, which by parity are equal as shown in Prop.~\ref{prop:analytic-remainder-parity-argument},
	\begin{align}
		\oint_{\partial\Omega_m} dw \frac{1}{w} \frac{1}{\left(\frac{d}{dw}J_q(w)\right)^2}
		& = I_{(R,m)}^{(den,j)} + I_{(T,m)}^{(den,j)} + I_{(L,m)}^{(den,j)} + I_{(B,m)}^{(den,j)} \\
		& = 0 + \pi i + 0 + \pi i = 2 \pi i
	\end{align}
	Hence, referring to Eq.~(\ref{eq:g-integral-power-series}), the analytic remainder for the denominator is just
	\begin{equation}\begin{aligned} \label{eq:bessel-evaluated-analytic-remainder-p2}
		g^{(den)}(z) & = \lim_{m \to \infty} \sum_{j=0}^\infty w^j \left( \frac{1}{2 \pi i} \oint_{\partial\Omega_m} dw \frac{1}{w^{j+1}}\frac{1}{\left(\frac{d}{dw}J_q(w)\right)^2} \right) \\
		& = \sum_{j=0}^\infty w^j \left( \frac{1}{2 \pi i} 2 \pi i \delta_{j,0} \right) = 1 \, .
	\end{aligned}\end{equation}
	
	\subsection{Bessel functions weighted at zeros with $p=1$} \label{sec:bessel-weighted-at-zeros-p=1}
	
	Next, consider the expansion of
	\begin{equation}
		J_q (z) = \frac{ \sfrac{J_q (z)}{J_q (z)} }{ \sfrac{1}{J_q (z)} } = \frac{1}{ \sfrac{1}{J_q (z)} } \, .
	\end{equation}
	As we will show in the remainder of this section, the result is
	\begin{equation}
		J_q(z) = \frac{1}{ A^{(0,den)}(z) + \sum_{k=-\infty,k\ne0}^\infty \frac{a_{-1}^{(k,den)}}{z-z_k} }
	\end{equation} 
	where $z_k$ are the zeros of the Bessel function aside from $z_0=0$, the Laurent principal part coefficients for $|k|\ge1$ are given by Eq.~(\ref{eq:bessel-laurent-simple-poles-weight-at-zero}), and the Laurent principal part for the pole at $z_0=0$ is
	\begin{equation}
		\begin{matrix}
			A^{(0,den)}(z) = \left\{\begin{matrix}
				\vphantom{0} \\
				\vphantom{\sum_{n=1}^{q} \frac{a_{-n}^{(0,den)}}{z^n}}
			\end{matrix}\right.
		\end{matrix}
		\begin{matrix*}[l]
			0 \, , & q = 0 \\
			\sum_{n=1}^{q} \frac{a_{-n}^{(0,den)}}{z^n} \, , & q \ge 1
		\end{matrix*}
	\end{equation}
	where the coefficients $a_{-n}^{(0,den)}$ are calculated by solving the matrix equation Eq.~(\ref{eq:laurent-coeff-mat-eq-simple-appendix}) with the values given in Sec.~\ref{sec:bessel-laurent-weight-at-zeros-p=1}.
	
	\subsubsection{Calculating the Laurent series principal part} \label{sec:bessel-laurent-weight-at-zeros-p=1}
	
	To obtain the Laurent coefficients for the pole at $z_0 = 0$, we again solve the matrix equation Eq.~(\ref{eq:laurent-coeff-mat-eq-simple-appendix}) for $q\ge1$. To account for the different denominator, this time we take $j_0 = 0$, $j_* = q$, $c_j = \delta_{j,0}$, and $b_{q+2j} = \frac{(-1)^j}{2^{q+2j}} \frac{1}{j!(q+j)!}$ for $j\ge0$ with all other terms vanishing.
	
	All other poles are simple and have
	\begin{equation} \label{eq:bessel-laurent-simple-poles-weight-at-zero}
		a_{-1}^{(k,den)} = \left.\frac{1}{\frac{d}{dz}J_q(z)}\right|_{z=z_k}
		= \frac{2}{J_{q-1}(z_k)-J_{q+1}(z_k)} \, .
	\end{equation}
	
	\subsubsection{Calculating the analytic remainder} \label{sec:bessel-analytic-remainder-zeros-p1}
	
	Here we only need to calculate the analytic remainder of the denominator using integrals of the forms
	\begin{equation}
		I_{(S,m)}^{(den,j)} = \int_{\Gamma_m^{(S)}} dw \frac{1}{w^{j+1}} \frac{1}{J_q(w)} \, ,
	\end{equation}
	similar to the denominator integrals of Eq.~(\ref{eq:bessel-p-1-integral-prototypes}). First consider the right side of the contour. Taking $R_m = \frac{\pi}{2}\left(2m+q+\frac{5}{2}\right)$, we have
	\begin{equation}
		\cos\left(R_m+iy-\frac{\pi}{2}q-\frac{\pi}{4}\right) = -(-1)^m \cosh y \, .
	\end{equation}
	Using this in the asymptotic form of Eq.~(\ref{eq:bessel-asymptotic}) and again expanding in a geometric series, we have
	\begin{align}
		\left.\frac{1}{J_q(w)}\right|_{R_m+iy}
		& = -\sqrt{\frac{\pi \left(R_m+iy\right)}{2}} \frac{1}{-(-1)^m \cosh y+e^{\left|y\right|}O\left(|R_m+iy|^{-1}\right)} \nonumber \\
		& = -\sqrt{\frac{\pi \left(R_m+iy\right)}{2}} \frac{(-1)^m}{\cosh y}\left(1 + O\left(R_m^{-1}\right)\right) \, ,
	\end{align}
	This is identical to the expansion for $\frac{1}{\frac{d}{dz}J_q(z)}$ given by Eq.~(\ref{eq:bessel-geometric-series-right-side}).
	Likewise, for the top side of the contour, we obtain an expansion identical to Eq.~(\ref{eq:bessel-remainder-1-num-top}),
	\begin{equation}
		\left.\frac{1}{J_q(w)}\right|_{x+iR_m} = e^{-R_m} O\left(R_m^{1/2}\right)
	\end{equation}
	Hence, all analytic remainder integrals of $\sfrac{1}{J_q(z)}$ are equal to those of $\sfrac{1}{\frac{d}{dz}J_q(z)}$, and so their analytic remainders are equal. Since the analytic remainder of $\sfrac{1}{\frac{d}{dz}J_q(z)}$ is $0$, so too is the analytic remainder of $\sfrac{1}{J_q(z)}$.
	
	\subsection{Bessel functions weighted at zeros with $p=2$} \label{sec:bessel-weighted-at-zeros-p=2}
	
	Next we expand
	\begin{equation}
		J_q (z) = \frac{ \sfrac{J_q (z)}{\left(J_q (z)\right)^2} }{ \sfrac{1}{\left(J_q (z)\right)^2} } = \frac{ \sfrac{1}{J_q (z)} }{ \sfrac{1}{\left(J_q (z)\right)^2} } \, .
	\end{equation}
	The numerator in this case was the denominator in Sec.~\ref{sec:bessel-weighted-at-zeros-p=1}; hence, we can reuse the results of those calculations. The only new calculations are the Laurent expansion and analytic remainder for the denominator.
	
	The resulting expansion is
	\begin{equation} \label{eq:bessel-weighting-zeros-2-final-result}
		J_q(z) = \frac{A^{(0,num)}(z) + \sum_{k=-\infty, k\ne0}^\infty \frac{a_{-1}^{(k,num)}}{z-z_k}}
		{1 + A^{(0,den)}(z) + \sum_{k=-\infty,k\ne0}^\infty \left(\frac{a_{-2}^{(k,den)}}{\left(z-z_k\right)^2} + \frac{a_{-1}^{(k,den)}}{z-z_k}\right) }
	\end{equation}
	where $z_k$ are the zeros of the Bessel function aside from $z_0=0$, the Laurent principal part coefficients for $|k|\ge1$ are given by Eq.~(\ref{eq:bessel-laurent-zeros-p2}), and the Laurent principal part for $k=0$ is
	\begin{equation}
		\begin{matrix}
			A^{(0,num)}(z) = \left\{\begin{matrix}
				\vphantom{0} \\
				\vphantom{\sum_{n=1}^{q} \frac{a_{-n}^{(0,num)}}{z^n}}
			\end{matrix}\right. \vspace{1.5mm} \\
			A^{(0,den)}(z) = \left\{\begin{matrix}
				\vphantom{0} \\
				\vphantom{\sum_{n=1}^{2q} \frac{a_{-n}^{(0,den)}}{z^n}}
			\end{matrix}\right.
		\end{matrix}
		\begin{matrix*}[l]
			0 \, , & q = 0 \\
			\sum_{n=1}^{q} \frac{a_{-n}^{(0,num)}}{z^n} \, , & q \ge 1 \vspace{1.5mm} \\
			0 \, , & q = 0 \\
			\sum_{n=1}^{2q} \frac{a_{-n}^{(0,den)}}{z^n} \, , & q \ge 1
		\end{matrix*}
	\end{equation}
	where $a_{-n}^{(0,num)}$ are calculated by solving the matrix equation Eq.~(\ref{eq:laurent-coeff-mat-eq-simple-appendix}) with the values given in Sec.~\ref{sec:bessel-laurent-weight-at-zeros-p=1}, and $a_{-n}^{(0,den)}$ are calculated by solving the matrix equation Eq.~(\ref{eq:laurent-coeff-matrix-equation-p2}) with the values given in Sec.~\ref{sec:bessel-laurent-weight-at-zeros-p=2}.
	
	\subsubsection{Calculating the Laurent series principal part} \label{sec:bessel-laurent-weight-at-zeros-p=2}
	
	To obtain the Laurent coefficients for the denominator's pole at $z_0 = 0$, we solve the matrix equation Eq.~(\ref{eq:laurent-coeff-matrix-equation-p2}) for $q\ge1$, but this time with the parameters of Sec.~\ref{sec:bessel-laurent-weight-at-zeros-p=1}, taking $j_0 = 0$, $j_* = q$, $c_j = \delta_{j,0}$, and from the Bessel function power series, $b_{q+2j} = \frac{(-1)^j}{2^{q+2j}} \frac{1}{j!(q+j)!}$ for $j\ge0$ with all other terms vanishing.
	
	Again writing $\frac{d^n}{dz^n}f(z) = f^{(n)}(z)$ for brevity, all of the denominator's other poles have
	\begin{equation} \label{eq:bessel-laurent-zeros-p2}
			a_{-2}^{(k,den)}
			= \frac{1}{(J_q^{(1)}(z_k))^2} \, , \quad
			a_{-1}^{(k,den)}
			= - \frac{J_q^{(2)}(z_k)}{(J_q^{(1)}(z_k))^3} \, .
	\end{equation}
	
	\subsubsection{Calculating the analytic remainder}
	
	In Sec.~\ref{sec:bessel-analytic-remainder-zeros-p1}, we showed that by taking $R_m = \frac{\pi}{2}\left(2m+q+\frac{5}{2}\right)$, we obtained analytic remainder integrals that were identical to those seen previously in Sec.~\ref{sec:bessel-analytic-remainder-p2}, and hence $\sfrac{1}{J_q (z)}$ has the same analytic remainder as $\sfrac{1}{\frac{d}{dz} J_q (z)}$. For the same reason, $\sfrac{1}{\left(J_q (z)\right)^2}$ has the same analytic remainder as $\sfrac{1}{\left(\frac{d}{dz} J_q (z)\right)^2}$, which is $1$.
	
	\subsection{Bessel functions weighted at both extrema and zeros with $p=1$} \label{sec:bessel-weighted-at-extrema-and-zeros-p=1}
	
	Finally, we expand
	\begin{equation}
		J_q (z)
		= \frac{ \sfrac{J_q (z)}{J_q (z)\frac{d}{dz}J_q(z)} }{ \sfrac{1}{J_q (z)\frac{d}{dz}J_q(z)} }
		= \frac{ \sfrac{1}{\frac{d}{dz}J_q(z)} }{ \sfrac{1}{J_q (z)\frac{d}{dz}J_q(z)} } \, .
	\end{equation}
	The Laurent principal parts and analytic remainder for the numerator have already been calculated in Sec.~\ref{sec:bessel-p-1}, so we only need to calculate these for the denominator.
	
	The resulting expansion is
	\begin{equation}
		J_q(z) = \frac{A^{(0,num)}(z) + \sum_{k=-\infty, k\ne0}^\infty \frac{a_{-1}^{(k,num)}}{z-z_k}}
		{A^{(0,den)}(z) + \sum_{k=-\infty,k\ne0}^\infty \frac{a_{-1}^{(k,den)}}{z-z_k} }
	\end{equation}
	where $z_k$ are the extrema and zeros of the Bessel function aside from $z_0=0$, the Laurent principal part coefficients for $|k|\ge1$ are given by Eq.~(\ref{eq:laurent-coeff-extrema-and-zeros}), and the Laurent principal part for $k=0$ is
	\begin{equation}
		\begin{matrix}
			A^{(0,num)}(z) = \left\{\begin{matrix}
				\vphantom{-\frac{2}{z}} \\
				\vphantom{0} \\
				\vphantom{\sum_{n=1}^{q-1} \frac{a_{-n}^{(0,num)}}{z^n}}
			\end{matrix}\right. \vspace{1.5mm} \\
			A^{(0,den)}(z) = \left\{\begin{matrix}
				\vphantom{-\frac{2}{z}} \\
				\vphantom{\frac{4}{z}} \\
				\vphantom{\sum_{n=1}^{2q-1} \frac{a_{-n}^{(0,den)}}{z^n}}
			\end{matrix}\right.
		\end{matrix}
		\begin{matrix*}[l]
			-\frac{2}{z} \, , & q = 0 \\
			0 \, , & q = 1 \\
			\sum_{n=1}^{q-1} \frac{a_{-n}^{(0,num)}}{z^n} \, , & q \ge 2 \vspace{1.5mm} \\
			-\frac{2}{z} \, , & q = 0 \\
			\frac{4}{z} \, , & q = 1 \\
			\sum_{n=1}^{2q-1} \frac{a_{-n}^{(0,den)}}{z^n} \, & q \ge 2
		\end{matrix*}
	\end{equation}
	where $a_{-n}^{(0,num)}$ are calculated by solving the matrix equation Eq.~(\ref{eq:laurent-coeff-mat-eq-simple-appendix}) with the values given in Sec.~\ref{sec:bessel-laurent-p1}, and $a_{-n}^{(0,den)}$ are calculated by solving the matrix equation Eq.~(\ref{eq:laurent-coeff-matrix-equation-extrema-and-zeros}) with the values given in Sec.~\ref{sec:bessel-laurent-weight-at-extrema-and-zeros-p=1}.
	
	\subsubsection{Calculating the Laurent series principal part} \label{sec:bessel-laurent-weight-at-extrema-and-zeros-p=1}
	
	As described in Sec.~\ref{sec:bessel-laurent}, the matrix equation for the pole at $z_0 = 0$ is
	\begin{equation} \label{eq:laurent-coeff-matrix-equation-extrema-and-zeros}
		c = B_0 B_1 a
	\end{equation}
	and the vector and matrix elements are given by
	\begin{align}
		\begin{matrix*}
			\left[a\right]_j & = & a_{-(j_{*0} + j_{*1} - j_0) + j} \, , \\
			\left[B_d\right]_{j,j'} & = & b_{j_{*d} + j - j'}^{(d)} \Theta\left(j-j'\right) \, , \\
			\left[c\right]_j & = & c_{j_0 + j} \, ,
		\end{matrix*}
		\left.\vphantom{\begin{matrix*}
				\left[a\right]_j & = & a_{-(2j_* - j_0) + j} \, , \\
				\left[B_d\right]_{j,j'} & = & b_{j_{*d} + j - j'}^{(d)} \Theta\left(j-j'\right) \, , \\
				\left[c\right]_j & = & c_{j_0 + j} \, ,
		\end{matrix*}}\right\}
		\quad
		\begin{matrix*}
			0 \le j,j' < j_{*0} + j_{*1} - j_0 - 1 \\
			d = 0, 1
		\end{matrix*}
	\end{align}
	where $j_0=0$, $j_{*0}=q$, $j_{*1}=q-1$, $c_j=\delta_{j,0}$, $b_{q+2j}^{(0)} = \frac{(-1)^j}{2^{q+2j}} \frac{1}{j!(q+j)!}$, and $b_{q-1+2j}^{(1)} = \frac{(-1)^j}{2^{q+2j}} \frac{q+2j}{j!(q+j)!}$ for $j\ge0$ with all other $b$ terms vanishing.
	
	All other poles are simple and their Laurent coefficients depend on whether $z_k$ is an extremum or zero of $J_q(z)$,
	\begin{equation}\begin{aligned} \label{eq:laurent-coeff-extrema-and-zeros}
		a_{-1}^{(k,den)} & = \frac{1}{\left(J_q^{(1)}(z)\right)^2} \quad \text{if} \quad J_q(z_k) = 0 \, , \\
		a_{-1}^{(k,den)} & = \frac{1}{\left(J_q^{(0)}(z)\right)\left(J_q^{(2)}(z)\right)} \quad \text{if} \quad \frac{d}{dz}J_q(z_k) = 0 \, .
	\end{aligned}\end{equation}
	
	\subsubsection{Calculating the analytic remainder}
	
	Here the integrals we must calculate have the form
	\begin{equation}
		I_{(S,m)}^{(den,j)} = \int_{\Gamma_m^{(S)}} dw \frac{1}{w^{j+1}} \frac{1}{J_q(w)\frac{d}{dw}J_q(w)} \, .
	\end{equation}
	To evaluate these, we combine both asymptotic forms in Eq.~(\ref{eq:bessel-asymptotic}) to find
	\begin{align} \label{eq:bessel-extrema-and-zeros-asymptotic}
		& J_q(w)\frac{d}{dw}J_q(w) \\
		& \quad\quad\quad = - \frac{2}{\pi w} \left(2\sin\left(2\left(w-\frac{\pi}{2}q-\frac{\pi}{4}\right)\right) + e^{2\left|Im(w)\right|} O\left(\left|w|^{-1}\right|\right)\right) \, . \nonumber
	\end{align}
	and we take $R_m = \tfrac{\pi}{2}\left(m+q+1\right)$, for which
	\begin{equation}
		\sin\left(2\left(R_m+iy-\frac{\pi}{2}q-\frac{\pi}{4}\right)\right) = \left(-1\right)^m \cosh 2y \, .
	\end{equation}
	From these, we obtain
	\begin{equation}
		\left.\frac{1}{J_q(w)\frac{d}{dw}J_q(w)}\right|_{R_m+iy}
		= - \frac{\pi \left(R_m+iy\right)}{4} \frac{(-1)^m}{\cosh 2y} \left(1+O\left(R_m^{-1}\right)\right)
	\end{equation}
	
	First we calculate the integral for the right side of the contour using the triangle inequality.
	\begin{equation}\begin{aligned} \label{eq:bessel-remainder-ex-ze-den-right}
		\left|I_{(R,m)}^{(den,j)}\right|
		& = \left|\int_{\Gamma_m^{(R)}} dw \frac{1}{w^{j+1}} \frac{1}{J_q(w)\frac{d}{dw}J_q(w)}\right| \\
		& \le \frac{\pi}{4} R_m^{-j} \int_{-R_m}^{R_m} dy \frac{1}{\cosh 2y} \left(1 + O\left(R_m^{-1}\right)\right) \\
		& = \frac{\pi}{4} R_m^{-j} \left(\arctan\sinh 2R_m\right)\left(1 + O\left(R_m^{-1}\right)\right) \, ,
	\end{aligned}\end{equation}
	which vanishes in the limit $m\to\infty$ for $j\ge1$. For $j=0$, we note that $\sfrac{1}{J_q(w)\frac{d}{dw}J_q(w)}$ is odd while $j$ is even, so that by the parity argument of Prop.~\ref{prop:analytic-remainder-parity-argument}, $I_{(R,m)}^{(num,j)}$ cancels with $I_{(L,m)}^{(num,j)}$ so that this integral never contributes to the analytic remainder.
	
	For the top side of the contour, we refer to Eq.~(\ref{eq:bessel-extrema-and-zeros-asymptotic}) to obtain
	\begin{equation}
		\left.\frac{1}{J_q(w)\frac{d}{dw}J_q(w)}\right|_{x+iR_m} = e^{-2R_m} O\left(R_m\right) \, .
	\end{equation}
	Hence the integral for the top side of the contour is
	\begin{equation}\begin{aligned}
		\left|I_{(T,m)}^{(den,j)}\right|
		= \left|\int_{\Gamma_m^{(T)}} dw \frac{1}{w^{j+1}} \frac{1}{J_q(w)\frac{d}{dw}J_q(w)}\right|
		= e^{-2R_m} O\left(R_m^{-j+1}\right) \, ,
	\end{aligned}\end{equation}
	which vanishes as $m\to\infty$ for all $j$.
	
	Since all contributions to the contour integrals either vanish in the limit $m\to\infty$ or are zero by parity, the analytic remainder for $\sfrac{1}{J_q(z)\frac{d}{dz}J_q(z)}$ vanishes.
	
	\subsection{Bessel functions weighted at both extrema and zeros with $p=2$} \label{sec:bessel-weighted-at-extrema-and-zeros-p=2}
	
	The form
	\begin{equation}
		J_q (z)
		= \frac{ \sfrac{J_q (z)}{\left(J_q (z)\frac{d}{dz}J_q(z)\right)
				^2} }{ \sfrac{1}{\left(J_q (z)\frac{d}{dz}J_q(z)\right)^2} }
	\end{equation}
	is not suitable for barycentric expansion using the methods of this paper, as its analytic remainder integrals for the numerator and denominator do not converge in the limit $m\to\infty$, violating the conditions of Corollary~\ref{cor:full-expansion-f}.
	
\end{document}